\documentclass{siamltex}
\usepackage[latin1]{inputenc}
\usepackage[english]{babel}
\usepackage{graphicx}
\usepackage{amsmath,amsfonts,amssymb}
\usepackage{bbm,wasysym,stmaryrd}
\usepackage{subfigure}
\usepackage{color}
\usepackage[normalem]{ulem}

\newcommand{\R}{\mathbb{R}}
\newcommand{\N}{\mathbb{N}}

\newcommand{\erf}{\mathrm{erf}}

\newcommand{\F}{\mathcal{F}}
\newcommand{\Pro}{\mathbb{P}}
\newenvironment{remark} {\par {\noindent \it \sc Remark.} \small \it } {}

\newcommand{\ta}{\tau_\alpha}
\newcommand{\Exp}[1]{\mathbb{E}\left [ #1 \right]}

\graphicspath{./}
\newcommand{\review}[1]{{{#1}}}

\title{Noise-induced behaviors in neural mean field dynamics}

\author{Jonathan Touboul\footnotemark[1]\ \footnotemark[3]\ \footnotemark[4] \and Geoffroy Hermann\footnotemark[5] \and Olivier Faugeras\footnotemark[5]}

\begin{document}
	
\maketitle

\renewcommand{\thefootnote}{\fnsymbol{footnote}}
\footnotetext[1]{{ The Mathematical Neuroscience Lab, Coll\`ege de France, Center of Interdisciplinary Research in Biology, 11 place Marcelin Berthelot, 75005 Paris,  CNRS UMR 7241m INSERM U1050, Universit\'e Pierre et Marie Curie ED 158. MEMOLIFE Laboratory of excellence and Paris Science Lettre. 11, place Marcelin Berthelot, 75005 Paris, France.
}}
\footnotetext[3]{{ BANG Laboratory, INRIA Rocquencourt, France
}}
\footnotetext[4]{{\tt jonathan.touboul@inria.fr}}
\footnotetext[5]{{ NeuroMathComp Laboratory, INRIA/ENS, 23 avenue d'Italie, 75013 Paris, France
}}

\begin{abstract}
The collective behavior of cortical neurons is strongly affected by the presence of noise at the level of individual cells. In order to study these phenomena in large-scale assemblies of neurons, we consider networks of firing-rate neurons with linear intrinsic dynamics and nonlinear coupling, belonging to a few types of cell populations and receiving noisy currents. Asymptotic equations as the number of neurons tends to infinity (mean field equations) are rigorously derived based on a probabilistic approach. These equations are implicit on the probability distribution of the solutions which generally makes their direct analysis difficult. However, in our case, the solutions are Gaussian, and their moments  satisfy a closed system of nonlinear ordinary differential equations (ODEs), which are much easier to study than the original stochastic network equations, and the statistics of the empirical process uniformly converge towards the solutions of these ODEs. Based on this description, we analytically and numerically study the influence of noise on the collective behaviors, and compare these asymptotic regimes to simulations of the network. We observe that the mean field equations provide an accurate description of the solutions of the network equations for network sizes as small as a few hundreds of neurons. In particular, we observe that the level of noise in the system qualitatively modifies its collective behavior, producing for instance synchronized oscillations of the whole network, desynchronization of oscillating regimes, and stabilization or destabilization of stationary solutions. These results shed a new light on the role of noise in shaping collective dynamics of neurons, and gives us clues for understanding similar phenomena observed in  biological networks. 
\end{abstract}


\begin{AMS}
34C15, % 	Nonlinear oscillations, coupled oscillators
34C23, %  	Bifurcation
60F99, % 	Limit theorems None of the above but in this section:
60B10, %	Convergence of probability measures
34C25  %	Periodic solutions
\end{AMS}

\pagestyle{myheadings}
\thispagestyle{plain}
\markboth{J. TOUBOUL, G. HERMANN \& O. FAUGERAS}{Noise effects in mean field dynamics}

\section*{Introduction}
The brain is composed of a very large number of neurons interacting in a complex nonlinear fashion and subject to noise. 
Because of these interactions, stimuli tend to produce coherent global responses, often with high reliability. 
%In neurobiology, variability is present at all scales, from the molecular level to the population level. 
At the scale of single neurons the presence of noise and nonlinearities often results in highly intricate behaviors. However, at larger scales, neurons form large ensembles that share the same input and are strongly connected, and at these scales, reliable responses to specific stimuli may arise. Such population assemblies (cortical columns or cortical areas) feature a very large number of neurons. 
Understanding the global behavior of these large-scale neural assemblies has been a focus of many investigations in the past decades. 
One of the main interests of large-scale modeling is to characterize brain function at the scale most non-invasive imaging techniques operate. 
Its relevance is  also connected to the fact that anatomical data recorded in the cortex reveal its columnar organization at scales ranging from about $50 \mu m$ to $1 mm$. These so-called cortical columns contain \review{on} the order of hundreds to hundreds of thousands of neurons, and are generally dedicated to specific functions. For example, in the visual cortex V1, they respond to preferred orientations in visual stimuli. 
These anatomical and functional organizations point towards the fact that a significant amount of information processing does not occur at the scale of individual neurons but rather corresponds to a mesoscopic signal arising from the collective dynamics of many interacting neurons.

In the computational neuroscience community, this problem has been mainly studied using approaches based on the statistical physics literature. Most models describing the emergent behavior arising from the interaction of neurons in large-scale networks have relied on continuum limits since the seminal works of Wilson and Cowan and Amari \cite{amari:72,amari:77,wilson-cowan:72,wilson-cowan:73}. Such models represent the activity of the network through a global variable, the population-averaged firing rate, which is generally assumed to be deterministic. Many analytical  properties and numerical results have been derived from these equations and related to cortical phenomena, for instance in the case of the problem of spatio-temporal pattern formation in spatially extended models (see e.g.~\cite{coombes-owen:05,ermentrout:98,ermentrout-cowan:79,bressloff-cowan-etal:02}). This approach implicitly makes the assumption that the effect of noise vanishes in large populations. 

However, increasingly many researchers now believe that the different intrinsic or extrinsic noise sources participate in the processing of information. Rather than having a pure disturbing effect there is the interesting possibility that noise conveys information and that this can be an important principle of brain function \cite{rolls-deco:10}. In order to study the effect of the stochastic nature of the firing in large networks, many authors strived to introduce randomness in a tractable form. A number of computational studies that successfully addressed the case of sparsely connected networks of integrate-and-fire neurons are based on the analysis of large assemblies that fire in an asynchronous regime~\cite{abbott-van-vreeswijk:93,amit-brunel:97,brunel-hakim:99}. Because of the assumption of sparse connectivity, correlations of the synaptic inputs can be neglected for large networks. The resulting asynchronous irregular states resemble the discharge activity recorded in the cerebral cortex of awake animals~\cite{destexhe:08}. 

Other models have been introduced to account for the presence of noise in neuronal networks, such as the population density method and related approaches \cite{cai-tao-etal:04}, allowing efficient simulation of large neuronal populations. In order to analyze the collective dynamics, most population density-based approaches involve expansions in terms of the moments of the resulting random variables, and the moment hierarchy needs to be truncated in order to get a closed set of equations, which can raise a number of technical issues (see e.g.\cite{ly-tranchina:07}). 

Yet other models of the activity of large networks are based on the definition of a Markov chain governing the firing dynamics of the neurons in the network, where the transition probability satisfies a differential equation called the \emph{master equation}. Seminal works of the application of such modeling for neuroscience date back to the early 90s and have been recently developed by several authors \cite{ohira-cowan:93,elboustani-destexhe:09}. Most of these approaches are proved correct in some parameter regions using statistical physics tools such as path integrals~\cite{buice-cowan:07} and Van-Kampen expansions~\cite{bressloff:09}. They motivated a number of interesting studies of quasicycles~\cite{bressloff:10} and power-law distribution of avalanche phenomena~\cite{benayoun-cowan-etal:10}. In many cases the authors consider one-step Markov chains, implying that at each update of the chain, only one neuron in the whole network either fires or stops firing, which raises biological plausibility issues. Moreover, analytical approaches mainly address the dynamics of a finite number of moments of the firing activity, which can also raise such issues as the well-posedness \cite{ly-tranchina:07} and the adequacy of these systems of equations with the original Markovian model~\cite{touboul-ermentrout:11}.

In the present study, we apply a probabilistic method to derive the limit behavior resulting from the interaction of an infinite number of firing-rate neurons nonlinearly interconnected. This approach differs from other works in the literature in that it relies on a description of the microscopic dynamics (neurons in the network), and does not make the assumption of a sparse connectivity. Our model takes into account the fact that cortical columns feature different populations. The approach consists in deriving the limit equations as the total number of neurons tends to infinity, based on results obtained in the field of large-scale systems of interacting particles. This problem has been chiefly studied for solving statistical physics questions, and has been a very active field of research in mathematics during the last decades~\cite{mckean:66,dobrushin:70,tanaka:78,sznitman:89}. In general, the equations obtained by such rigorous approaches are extremely hard to analyze. They can be either seen as  implicit equations in the set of stochastic processes, or as non-local partial differential equations on the probability distribution through the related Fokker-Planck equations. But in both cases, understanding the dynamics of these equations is very challenging, even for basic properties such as the existence and uniqueness of stationary solutions and a priori estimates~\cite{herrmann:09}. It appears even more difficult to understand qualitatively the effects of noise on the solutions and to interpret them in terms of the underlying biological processes. 

In the present article we aim at answering some of these questions.  In the case we address, the problem is rigorously reducible to the analysis of a set of ordinary differential equations. This is because the solution of the mean field equations is a Gaussian process. It is therefore completely determined by its  first two moments which we prove to be the solutions of ordinary differential equations. This  allows us to go much deeper into the analysis of the dynamical effects of the parameters, in particular those related to the noise, and to understand their influence on the solutions. The analysis of this Gaussian process also provides a rich amount of  information about the  non-Gaussian solution of the network when its size is large enough.  

The paper is organized as follows. In the first section we deal with the modeling, the derivation of the mean field equations and of the related system of ordinary differential equations. We then turn in section \ref{sec:bifs} to the analysis of the solutions of these equations and the influence of noise. We show in details how noise strongly determines the activity of the cortical assembly. We then return to the  problem of understanding the behavior of finite-size (albeit large) networks in section~\ref{sec:Network} and compare their behavior with those of the solutions of the mean field equations (infinite-size network). The analysis of the network behaviors in the different regimes of the mean field equations provides an interpretation of the individual behaviors responsible for collective reliable responses. This has a number of consequences that are developed in section~\ref{sec:discussion}.

\section{Model and mean field equations}\label{sec:Models}
In all the article, we work in a complete probability space $(\Omega,\F,\Pro)$ assumed to satisfy the usual conditions. 

We are interested in the large scale behavior arising from the nonlinear coupling of a large number $N$ of stochastic diffusion processes representing the membrane potential of neurons in the framework of rate models (see e.g.~\cite{dayan-abbott:01,gerstner-kistler:02b}). Hence the variable characterizing the neuron state is its firing rate, that exponentially relaxes to zero when it receives no input, and that integrates both external input and the current generated by its neighbors. The network is composed of $P$ neural populations that differ by their intrinsic dynamics, the input they receive and the way they interact with the other neurons. Each population $\alpha \in \{1,\ldots,P\}$ is composed of $N_{\alpha}$ neurons, and we assume that the ratio $N_{\alpha}/N$ converges to a constant $\delta_{\alpha}$ in $]0,1[$ when the total number of neurons $N$ becomes arbitrarily large. We define the population function $p$ that maps the index $i\in\{1,\ldots N\}$ of any neuron to the index  $\alpha$ of the population neuron $i$ belongs to: $p(i)=\alpha$.

For any neuron $i$ in population $\alpha$, the membrane potential $V^i_t$ has a linear intrinsic dynamics with a time constant $\tau_{\alpha}$. The membrane potential of each neuron returns to zero exponentially if it receives no input. The neuron $i$ in population $\alpha$ receives an external current, which is the sum of a deterministic part $I_{\alpha}(t)$ and a stochastic additive noise driven by $N$ independent adapted Brownian motions $(B^i)_{i=1\ldots N}$ modulated by the diffusion coefficients $\lambda_{\alpha}(t)$. This additive noise term accounts for different biological phenomena~\cite{aldo-faisal:08}, such as intrinsically noisy external input, channel noise~\cite{white-rubinstein:00} produced by the random opening and closing of ion channels, thermal noise and thermodynamic noise related to the discrete nature of charge carriers.

Neurons also interact through their firing rates, classically modeled as a sigmoidal transform of their membrane potential. These sigmoidal functions only depend on the population $\alpha$ the neuron $i$ belongs to: $S_{\alpha}(V^i)$. The functions $S_{\alpha}:\R\mapsto \R$ are assumed to be smooth (Lipchitz continuous), increasing functions that tend to $0$ at $-\infty$ and to $1$ at $\infty$. The  firing rate of the presynaptic neuron $j$, multiplied by the synaptic weight $J_{ij}$, is an input current to the postsynaptic neuron $i$. We classically assume that the synaptic weight $J_{ij}$ is equal to $J_{p(i)p(j)}/N_{p(j)}$. In practice this synaptic weight randomly varies depending on the local properties of the environment. Models including this type of randomness are not covered in this paper. The scaling assumption is necessary to  ensure that the total input to a neuron does not depend on the network size. 

The network behavior is therefore governed by the following set of stochastic  differential equations:
\begin{equation}\label{eq:Network}
dV^{i}(t) = \left( -\frac 1 {\ta}  V^{i}(t) + I_{\alpha}(t) + \sum_{\beta=1}^{P} \frac{J_{\alpha\beta} }{N_{\beta}} \sum_{j:\,p(j)=\beta} S_{\beta}(V^j(t)) \right) \, dt +
\lambda_{\alpha}(t) dB^{i}_t 
\end{equation}

These equations represent a set of interacting diffusion processes. Such processes have been studied for instance by McKean, Tanaka and Sznitman among others~\cite{mckean:66,tanaka:83,tanaka:78,sznitman:89}. This case is simpler than the different cases treated in~\cite{touboul-faugeras:11} where the intrinsic dynamics of each individual diffusion is nonlinear. 

\review{We aim at identifying the limit in law of the activity of finite sets of neurons in the network as the number of neurons tend to infinity. The identification of this law will imply the fact that the system satisfies the \emph{propagation of chaos} property. This property states that, provided the initial conditions are independent and identically distributed for all neurons of each population (initial conditions are said to be \emph{chaotic}\footnote{\review{Note that the term \emph{chaos} is understood here in the statistical physics sense as the Boltzmann's molecular chaos ("Sto\ss zahlansatz"), corresponding to the independence between the velocities of two different particles before they collide. This is very different from the notion of chaos in deterministic dynamical systems.} }), then in the limit $N\to\infty$, all neurons behave independently, and have the same law which is given by an implicit equation on the law of the limiting process (the \emph{chaos} of the initial condition is propagated for all time $t>0$). In details, the law of $(V^{i_1}(t), \ldots, V^{i_k}(t), t\leq T)$ for any fixed $k\geq 1$ and $(i_1,\ldots,i_k) \in \N^k$, converges towards $\nu_{p(i_1)}\otimes\ldots\otimes\nu_{p(i_k)}$ when $N\to \infty$, where $\nu_{\alpha}$ denotes the law of the solution of equation \eqref{eq:MFE} corresponding to population $\alpha$.}

This convergence and the limit equations are the subject of the following theorem:

\begin{theorem}\label{thm:propagationchaos}
	Let $T>0$ a fixed time. Under the previous assumptions, we have:
	\renewcommand{\theenumi}{(\roman{enumi})}
	\begin{enumerate}
		\item The process $V^i$ for $i$ in population $\alpha$, solution of equation \eqref{eq:Network}, converges in law towards the process $\bar{V}^{\alpha}$ solution of the mean field implicit equation:
			\begin{multline}\label{eq:MFE}
			d\bar{V}^{\alpha}(t)=\left[-\frac 1 {\ta}  \bar{V}^{\alpha}(t)+I_{\alpha}(t)+\sum_{\beta=1}^P J_{\alpha\beta} \Exp{S_{\beta}(\bar{V}^{\beta}(t))}\right]dt
			+\lambda_{\alpha}(t) dB^{\alpha}(t)
			\end{multline}
			as a process for $t \in [0,T]$, in the sense that there exists $(\bar{V}_t^i)_{t\geq 0}$ distributed as $(\bar{V}^{\alpha}_t)_{t\geq 0}$ such that
			\[\Exp{\sup_{0\leq t\leq T} \vert V^i_t - \bar{V}^i_t\vert }\leq \frac {\tilde{C}(T)}{\sqrt{N}}\]
			where $\tilde{C}(T)$ is a \review{finite quantity depending on the time horizon $T$ and} on the parameters of the system. As a random variable, it converges uniformly in time in the sense that:
			\[\sup_{0\leq t\leq T} \Exp{\vert V^i_t - \bar{V}^i_t\vert }\leq \frac {C}{\sqrt{N}}\]
			where $C$ does not depend on time. In equations \eqref{eq:MFE}, the processes $(B^{\alpha}(t))_{\alpha=1 \ldots P}$ are independent Brownian motions.
		\item Equation \eqref{eq:MFE} has a unique (trajectorial and in law) solution which is square integrable.
		\item The propagation of chaos applies.
	\end{enumerate}
\end{theorem}
\review{Note that the expectation term in equation~\eqref{eq:MFE} is the classical expectation of a function of a stochastic process. In details, if $p^{\beta}_t$ is the probability density of $\bar{V}^{\beta}(t)$, $\Exp{S_{\beta}(\bar{V}^{\beta}(t))}$ is equal to $\int_{\R} S_{\beta}(x)p^{\beta}_t(x)\,dx$. }

The proof of this theorem essentially uses results from the works of Tanaka and Sznitman, summarized in~\cite{sznitman:89}. A distinction with these classical results is that the network is not totally homogeneous but composed of distinct neural populations.  Thanks to the assumption that the proportion of neurons in each population is non-trivial ($N_{\alpha}/N \to \delta_{\alpha}\in ]0,1[$), the propagation of chaos occurs simultaneously in each population yielding our equations, as it is shown in a wider setting in~\cite{touboul-faugeras:11}. The main deep theoretical distinction is that the theorem claims a uniform convergence in time: most of the results proved in the kinetic theory domain show propagation of chaos properties and convergence results only for finite time \review{intervals}, and convergence estimates diverge as the \review{interval grows}. Uniform propagation of chaos is an important property as commented in~\cite{malrieu:08}, and particularly in our case as we will further comment. Methods to prove uniformity are generally involved (see e.g.~\cite{mischler-mouhot:11} where uniformity is obtained for certain models using a dual approach based on the analysis of generator operators). Due to the linearity of the intrinsic dynamics, we provide here an elementary proof of this property in our particular system.
\begin{proof}
The existence and uniqueness of solutions can be performed in a classical fashion using Picard iterations of an integral form of equation \eqref{eq:MFE} and a contraction argument. The proof of the convergence towards this law, and of the propagation of chaos can be performed using Sznitman's powerful coupling method (see e.g.~\cite{sznitman:89}, method which dates back from Dobrushin~\cite{dobrushin:70}), that consists in \review{exhibiting an almost sure limit of the sequence of processes $V^i_t$ as $N$ goes to infinity by coupling the mean field equation with the network equation as follows. We define the} different independent processes $\bar{V}^i$ solution of equation \eqref{eq:MFE} driven by the same Brownian motion $(B^i_t)_t$ as involved in the network equation \eqref{eq:Network}, and with the same initial condition $V^i(0)$ as neuron $i$ in the network. \review{It is clear that these processes are independent (since the $(B^i_t)$ are pairwise independent) and have the same law as the solution of the mean field equation~\eqref{eq:MFE}. The almost sure convergence of $(V^i_t)$ towards $(\bar{V}^i_t)$ will therefore imply the convergence in law towards the mean field equation.} For a neuron $i$ belonging to population $\alpha$, we have:
\begin{multline*}
	V^i_t - \bar{V}^i_t=\sum_{\beta=1}^P J_{\alpha\beta} \int_0^t e^{-(t-s)/\tau_{\alpha}} \frac{1}{N_{\beta}} \sum_{j:\,p(j)=\beta} \Bigg\{ \Big(S_{\beta}(V^j_s)-S_{\beta}(\bar{V}^j_s)\Big) \\
	+ \Big(S_{\beta}(\bar{V}^j_s) - \Exp{S_{\beta}(\bar{V}^j_s)}\Big)\Bigg\}\, ds
\end{multline*}

We have, denoting by $\tau$ the maximal value of $(\tau_{\beta}, \beta=1\ldots P)$:
\begin{multline}\label{eq:vimoinsvibar}
\vert V^i_t - \bar{V}^i_t \vert \leq K_{\alpha} \int_0^t e^{-(t-s)/\tau} \max_{j=1\ldots N}\,\vert V^j_s - \bar{V}^j_s   \vert\,ds \\ 
	+ K_{\alpha}' \left\vert \frac{1}{N} \int_0^t e^{-(t-s)/\tau_{\alpha}} \sum_{j=1}^N \big(S_{p(j)}(\bar{V}^j_s) - \Exp{S_{p(j)}(\bar{V}^j_s)}\big)\, ds\right\vert,
\end{multline}
where $K_{\alpha}=\sum_{\beta}\vert J_{\alpha\beta}\vert L $. $L$ is the largest Lipschitz constant of the sigmoids $(S_{\beta}, \beta=1\ldots P)$, and $K_{\alpha}'=\max_{\beta}\vert J_{\alpha\beta}\vert N/N_{\beta}$, quantity upperbounded, for $N$ sufficiently large, by $\max_{\beta}\vert J_{\alpha\beta}\vert 2/\delta_{\beta}$. 

Since the righthand side of \eqref{eq:vimoinsvibar} does not depend on the index $i$, taking the maximum with respect to $i$ and the expected value of both sides of \eqref{eq:vimoinsvibar},  we obtain
\begin{multline}\label{eq:vimoinsvibarineq}
\Exp{ \max_{i=1\ldots N}\vert V^i_t - \bar{V}^i_t \vert} \leq  K \int_0^t e^{-(t-s)/\tau} \Exp{ \max_{j=1\ldots N}\,\vert V^j_s - \bar{V}^j_s   \vert }\,ds\\
+K'\Exp{ \max_{\alpha=1,\ldots, P}\left\vert \frac{1}{N} \int_0^t e^{-(t-s)/\tau_{\alpha}} \sum_{j=1}^N \big(S_{p(j)}(\bar{V}^j_s) - \Exp{S_{p(j)}(\bar{V}^j_s)}\big)\, ds\right\vert }
\end{multline}
Since the random variables $A_j(s)=S_{p(j)}(\bar{V}^j_s) - \Exp{S_{p(j)}(\bar{V}^j_s)}$ are independent and centered \review{(i.e. have a null expectation)}, using the fact that the sigmoids $S_{\beta}$ take their values in the interval $[0,1]$, using Cauchy-Scwartz and posing $\bar{\tau}=\max_{\alpha}\tau_{\alpha}+1=\tau+1$, we have:
\begin{align*}
	&\Exp{\max_{\alpha} \left( \frac{1}{N} \int_0^t e^{-(t-s)/\tau_{\alpha}}\sum_{j=1}^N \big(S_{p(j)}(\bar{V}^j_s) - \Exp{S_{p(j)}(\bar{V}^j_s)}\big)\,ds \right)^2} \\
        &\qquad \qquad = \frac 1 {N^2}\Exp{\max_{\alpha} \left(\int_0^t\left( e^{-(t-s)\frac{\bar{\tau}-\tau_{\alpha}}{\bar{\tau}\tau_{\alpha}}}\right)\left(\sum_{j=1}^N e^{-(t-s)/\bar{\tau}}A_j(s) \right)\right)^2 ds} \\
        &\qquad \qquad \leq \frac 1 {N^2} \Exp{\max_{\alpha} \left(\int_0^t e^{-2(t-s)\frac{\bar{\tau}-\tau_{\alpha}}{\bar{\tau}\tau_{\alpha}}}ds\right)\left(\int_0^t e^{-2(t-s)/\bar{\tau}} \big(\sum_{j=1}^N A_j(s)\big)^2 ds\right)}\\
        &\qquad \qquad \leq \frac 1 {N^2} \Exp{\left(\int_0^t e^{-2(t-s)/(\tau(\tau+1))}ds\right)\left(\int_0^t e^{-2(t-s)/(\tau+1)} \big(\sum_{j=1}^N A_j(s)\big)^2 ds\right)}\\
        &\qquad \qquad = \frac 1 {N^2} \frac{\tau(\tau+1)}{2} (1-e^{-2t/(\tau(\tau+1))}) \int_0^t e^{-2(t-s)/(\tau+1)}\Exp{\sum_{j=1}^N A_j(s)^2}ds\\
        &\qquad \qquad \leq \frac 1 {N} \frac{\tau(\tau+1)}{2} \int_0^t e^{-2(t-s)/(\tau+1)}ds\\
	&\qquad \qquad \leq \frac 1 {N} \frac{\tau(\tau+1)^2}{4}=\frac{\tau'}{N}
\end{align*}

% Defining $M_t=\sup_{i} \Exp{\sup_{0\leq s\leq t} \vert V^i_s - \bar{V}^i_s\vert}$, we obtain using simple algebra:
% By Cauchy-Schwartz inequality, we can  upperbound the second term of the righthand side of equation \eqref{eq:vimoinsvibar} by $\tau/\sqrt{N}$. Therefore, defining $M_t=\Exp{\sup_{i}  \vert V^i_t - \bar{V}^i_t\vert}$, we have:
% \[M_t\leq K \int_0^t e^{-(t-s)/\tau} M_s \,ds + K'\frac{\tau} {\sqrt{N}}\]

% Defining $M_t=\sup_{i} \Exp{\sup_{0\leq s\leq t} \vert V^i_s - \bar{V}^i_s\vert}$, we obtain using simple algebra:
By Cauchy-Schwartz inequality, we can  upperbound the second term of the righthand side of inequality \eqref{eq:vimoinsvibarineq} by $\sqrt{\tau'/N}$. Therefore, defining $M_t=\Exp{\max_{i}  \vert V^i_t - \bar{V}^i_t\vert}$, $K=\max_{\alpha} K_{\alpha}$ and $K'=\max_{\alpha} K_{\alpha}'$ we have:
\[M_t\leq K \int_0^t e^{-(t-s)/\tau} M_s \,ds + K'\sqrt{\frac{\tau'}{N}}\]
implying, using Gronwall's lemma, 
\[M_t \leq \frac{K'\sqrt{\tau'} e^{K\tau}}{\sqrt{N}}.\]
This inequality readily yields the almost sure convergence of $V^i_t$ towards $\bar{V}^i_t$ as $N$ goes to infinity, uniformly in time, and hence convergence in law of $V^i_t$ towards $\bar{V}^{\alpha}_t$.

The almost sure convergence of $(V^i_t)_{t\in [0,T]}$ (considered as a process) towards $(\bar{V}^i_t)_{t\in [0,T]}$ can be proved in a similar fashion. Indeed, upperbounding the exponential term in \eqref{eq:vimoinsvibar} by $1$ and taking the supremum, it is easy to see that:
\[\Exp{\sup_{0\leq t\leq T} \max_{i=1\ldots N} \vert V^i_t-\bar{V}^i_t\vert} \leq K \int_0^T \Exp{\sup_{s\in[0,t]} \max_{j=1\ldots N}\,\vert V^j_s - \bar{V}^j_s   \vert}\,dt + \frac{K'\,T} {\sqrt{N}},\]
using the fact that:
\begin{align*}
	&\Exp{ \max_{\alpha}\sup_{t\in [0,T]} \left( \frac{1}{N} \int_0^t e^{-(t-s)/\tau_{\alpha}}\sum_{j=1}^N \big(S_{p(j)}(\bar{V}^j_s) - \Exp{S_{p(j)}(\bar{V}^j_s)}\big)\,ds \right)^2} \\
	&\qquad \qquad \leq \frac T {N^2}  \int_0^T\Exp{\left \vert \sum_{j=1}^N \big(S_{p(j)}(\bar{V}^j_s) - \Exp{S_{p(j)}(\bar{V}^j_s)}\big) \right\vert^2}\,ds\\
	&\qquad \qquad = \frac T {N^2} \sum_{j=1}^N \int_0^T\Exp{ \Big \vert S_{p(j)}(\bar{V}^j_s) - \Exp{S_{p(j)}(\bar{V}^j_s)}\Big\vert ^2}\,ds\\
	&\qquad \qquad \leq \frac {T^2} {N}
\end{align*}
using the independence of the $\bar{V}^j$ and Cauchy-Schwartz inequality. This last estimate readily implies, using Gronwall's inequality:
\[\Exp{\sup_{0\leq t\leq T} \max_{i=1\ldots N} \vert V^i_t-\bar{V}^i_t\vert} \leq \frac{K'\,T e^{K\,T}}{\sqrt{N}}.\]
 
The propagation of chaos property (iii) stems from the almost sure convergence of $(V^{i_1}(t), \ldots, V^{i_k}(t), t\leq T)$ towards $(\bar{V}^{i_1}(t), \ldots, \bar{V}^{i_k}(t), t\leq T)$ \review{which are independent}, as a process and uniformly for fixed time, and is proved in a similar fashion.
\end{proof}

The $P$ equations \eqref{eq:MFE}, which are  $P$ implicit stochastic differential equations, describe the asymptotic behavior of the network. However, the characterization and simulation of their solutions is a challenge. Fortunately, due to their particular form in our setting, these equations can be substantially simplified. \review{Indeed, under some weak assumptions, the solutions of the mean field equations are shown to be Gaussian, allowing to exactly reduce the dynamics of the mean field equations to the study of coupled ordinary differential equations as we now show.}

\begin{proposition}\label{pro:gaussianSolution}
Let us assume that \review{$\bar{V}_0=(\bar{V}_0^{\alpha})_{\alpha=1...P}$ } is a P-dimensional Gaussian random variable. We have:
\renewcommand{\labelitemi}{$\star$}
	\begin{itemize}
		\item The solutions of the P mean field equations \eqref{eq:MFE} with initial conditions \review{$\bar{V}_0$} are Gaussian processes for all time.
		\item Let $\mu(t)=(\mu_{\alpha}(t))_{\alpha=1\ldots P}$ denote the mean vector of \review{$(\bar{V}^{\alpha}_t)_{\alpha=1\ldots P}$} and $v(t)=(v_{\alpha}(t))_{\alpha=1\ldots P}$ its variance. Let also $f_{\beta}(x,y)$ denote the expectation of $S_{\beta}(U)$ for  $U$ a Gaussian random variable of mean $x$ and variance $y$. We have:
			\begin{equation}\label{eq:ODE}
			\begin{cases}
				\dot{\mu}_{\alpha}(t)=-\frac 1 {\ta} \mu_{\alpha}(t) + \sum_{\beta=1}^P J_{\alpha\beta}f_{\beta}(\mu_{\beta}(t),v_{\beta}(t))+I_{\alpha}(t) & \alpha=1\ldots P\\
				\dot{v}_{\alpha}(t)=-\frac 2 {\ta} \,v_{\alpha}(t) +\lambda^2_{\alpha}(t)  & \alpha=1\ldots P
			\end{cases}
			\end{equation}
			with initial conditions \review{$\mu_{\alpha}(0) =\Exp{\bar{V}_0^{\alpha}}$ and $v_{\alpha}(0) =\Exp{(\bar{V}_0^{\alpha}-\mu_{\alpha}(0))^2}$}. In equation \eqref{eq:ODE}, the dot denotes the differential with respect to time. 
	\end{itemize}
\end{proposition}	

\begin{proof}
	The unique solution of the mean field equations \eqref{eq:MFE} starting from a square integrable initial condition $\bar{V}_0$ measurable with respect to $\F$ can be written in the form:
		\begin{multline}\label{eq:SoluMFE}
		\review{\bar{V}}^{\alpha}_t=e^{-\frac{t}{\ta}} \review{\bar{V}^{\alpha}_0}+e^{-\frac{t}{\ta}}\Big(\int_0^t e^{\frac{s}{\ta}} (I_{\alpha}(s)+\sum_{\beta=1}^P J_{\alpha\beta}\Exp{S_{\beta}(\review{\bar{V}}^{\beta}_s)}) 
		ds\\+ \int_0^t e^{\frac{s}{\ta}}\lambda_{\alpha}(s)\, dB^{\alpha}_s\Big).
		\end{multline}
		% This fixed point equation is shown in~\cite{touboul-faugeras:11} to have a unique solution, which is also the unique limit of the iteration of the map $\Phi=(\Phi(\cdot)^{\alpha}, \alpha=1,\cdots,P)$, defined on the space of stochastic processes:
		% \begin{multline*}
		% 	\Phi(X)_t^{\alpha}=e^{-\frac{t}{\ta}} X_{\alpha}^0+e^{-\frac{t}{\ta}}\Big(\int_0^t e^{\frac{s}{\ta}} (I_{\alpha}(s)+\sum_{\beta=1}^P J_{\alpha\beta}\Exp{S_{\beta}(X_{\beta}(s))}) 
		% ds\\+\int_0^t e^{\frac{s}{\ta}}\lambda_{\alpha}(s)\, dB^{\alpha}_s\Big).
		% \end{multline*}
		% In other words, for any square-integrable stochastic process $X$, the unique solution of the mean field equations is equal in law to the limit of $\Phi^n(X)$ as $n$ goes to infinity. 
		\review{We observe that if  $\bar{V}_0$ is a Gaussian random variable, then the righthand side of~\eqref{eq:SoluMFE} is necessarily Gaussian as the sum of a scaled version of this random variable, a deterministic term and an It\^o integral of a deterministic function, and hence so is the solution of the mean field equation.}
		% also Gaussian whatever $X$. We hence conclude that the unique solution of the mean field equation is Gaussian, as a limit in law of a sequence of Gaussian processes. 
	% 
		Its law is hence characterized by its mean and covariance functions. The formula \eqref{eq:SoluMFE} involves the expectation $\Exp{S_{\beta}(\review{\bar{V}}^{\beta}(s))}$, which, because of the Gaussian nature of $X_{\beta}$, only depends on $\mu_{\beta}(s)$ and $v_{\beta}(s)$, and is denoted by $f_{\beta}(\mu_{\beta}(s),v_{\beta}(s))$. Taking the expectation of both sides of the equality \eqref{eq:SoluMFE}, we obtain the equation satisfied by the mean of the process $\mu_{\alpha}(t)=\Exp{\review{\bar{V}}^{\alpha}(t)}$:
	\[
	 \mu_{\alpha}(t)=e^{-\frac{t}{\tau_{\alpha}}}\left(\review{\mu^{\alpha}(0)}+\int_0^t e^{\frac{s}{\tau_{\alpha}}} \left(\sum_{\beta=1}^P J_{\alpha\beta}\Exp{S_{\beta}(\review{\bar{V}}^{\beta}_s}+I_{\alpha}(s) \right) ds\right).
	\]
	Taking the variance of both sides of the equality \eqref{eq:SoluMFE}, we obtain the following equation:
	\[
	v_{\alpha}(t)=	e^{-\frac{2t}{\ta}}\left(v_{\alpha}(0)+\int_0^t e^{\frac{2s}{\ta}}\lambda_{\alpha}^2(s)\, ds\right),
	\]
and this concludes the proof. 
\end{proof}

\begin{remark}
	\begin{itemize}
		\item In order to fully characterize the law of the process $\bar{V}$, we need to compute the covariance matrix function \review{$\textrm{Cov}(\bar{V}^{\alpha}_{t_1},\bar{V}^{\beta}_{t_2})$ for $t_1$ and $t_2$ in $\R^{+*}$. For $\alpha\neq \beta$ this covariance is clearly zero using equation~\eqref{eq:SoluMFE}, and we have:}
		\begin{equation}\label{eq:Covariance}
			\review{\textrm{Cov}(\bar{V}^{\alpha}_{t_1},\bar{V}^{\alpha}_{t_2}) = e^{-\frac{t_1+t_2}{\ta}}\left( \textrm{Var}(\bar{V}_{\alpha}^0)
			+ \int_0^{t_1\wedge t_2} e^{\frac{2s}{\ta}}\lambda_{\alpha}^2(s)\,ds\right)}
		\end{equation}
		for $t_1,t_2 \in \R^{+*}$, hence only depends on the parameters of the system and is in particular not coupled to the dynamics. The description of the solution given by equations \eqref{eq:ODE} is hence sufficient to fully characterize the solution of the mean field equations \eqref{eq:MFE}.
		\item The uniformity in time of the propagation of chaos has deep implications in regard of equations~\eqref{eq:ODE}. Indeed, we observe that the solution of the mean field equation is governed by the mean of the process, the expectation being a deterministic function depending on the parameters of the system. The uniformity in particular implies that, for $i$ in population $\alpha$:
		\begin{equation}\label{eq:UniformConvergenceMean}
			\sup_{t\geq 0} \left \vert \Exp{V^i_t} - \mu_{\alpha}(t) \right \vert \leq \sup_{t\geq 0}\Exp{\vert V^i_t-\bar{V}^i_t\vert} \leq \frac C {\sqrt{N}}
		\end{equation} 
		implying uniform convergence of the empirical mean, as a function of time, towards $\mu_{\alpha}(t)$.
		\item If \review{$\bar{V}_0$} is not Gaussian, the solution of equation \eqref{eq:MFE} asymptotically converges exponentially towards a Gaussian solution. The important uniformity convergence property towards the mean field equations ensures that the Gaussian solution is indeed the asymptotic regime of the network, which strengthens the interest of the analysis of the differential system \eqref{eq:ODE}.
	\end{itemize}
\end{remark}

The functions $f_{\beta}$ depend on the choice of the sigmoidal transform. A particularly interesting choice is the $\erf$ sigmoidal function\footnote{\review{We emphasize that in this article, the $\erf$ function is understood as the repartition function of a standard Gaussian random variable, i.e. $\erf(x) = \int_{-\infty}^x \frac{\exp\{-x^2/2\}}{\sqrt{2\pi}}$. In particular, this function is always positive. It is distinct from the usual definition of the $\erf$ function corresponding to $Erf(x)=\frac{2}{\sqrt{\pi}}\int_{0}^x e^{-x^2}$. This function is not always positive (hence not a good rate function) and is related to our function through the formula: $\erf(x)=\frac{1}{2} (Erf(\frac{x}{\sqrt{2}})+1)$. }}  $S_{\alpha}(x) = \erf(g_{\alpha}x+\gamma_{\alpha})$. In that case we are able to express the function $f_{\beta}$ in closed form, because of the following lemma:

\begin{lemma}\label{lemma:ErfSigmoids}
	In the case where the sigmoidal transforms are of the form $S_{\alpha}(x) = \erf(g_{\alpha}x+\gamma_{\alpha})$, the functions $f_{\alpha}(\mu_{\alpha},v_{\alpha})$ involved in the mean field equations \eqref{eq:ODE} with a Gaussian initial condition take the simple form:
	\begin{equation}\label{eq:M-Erf}
		f_\alpha(\mu,v) = \erf\left(\frac{g_{\alpha}\, \mu + \gamma_{\alpha}}{\sqrt{1+g_\alpha^{2}v}}\right).
	\end{equation}
\end{lemma}
\begin{proof}
The proof is given in Appendix \ref{app:proof}.
\end{proof}

In summary, we have shown that, provided that the initial conditions of each neuron are independent and Gaussian, the large-scale behavior of our linear model is governed by a set of ordinary differential equations (theorem~\ref{thm:propagationchaos} and proposition \ref{pro:gaussianSolution}). This is very interesting since it reduces the study of the solutions to the very complex implicit equation \eqref{eq:MFE} bearing on the law of a process to a much simpler setting, ordinary differential equations. As shown below this allows us to understand the effects of the system parameters on the solutions. For this reason we assume from now on that the initial condition is Gaussian, and focus on the effect of the noise on the dynamics.

\section{Noise-induced phenomena}\label{sec:bifs}
In this section we mathematically and numerically study the influence of the noise levels $\lambda_{\alpha}$ on the dynamics of the neuronal populations. Thanks to the uniform convergence of the empirical mean towards the mean of the mean field system (equation~\eqref{eq:UniformConvergenceMean}) and the propagation of chaos property for the network process, it is relevant to study such phenomena through the thorough analysis of the solutions of the mean field equations given by the ODEs~\eqref{eq:ODE}. This is what we do in the present section.
%For simpler cases, we will be able to develop analytical characterization. We will then turn to a numerical analysis of the bifurcations of the system in a more general case. 

As observed in equation~\eqref{eq:ODE}, in the case of a Gaussian initial condition, the  equation of the variance $v$ is decoupled from the equation on the mean $\mu$ in \eqref{eq:ODE}. The variance satisfies an autonomous equation:
\[\dot{v}_{\alpha}=-\frac{2}{\ta} v_{\alpha} + \lambda^2_{\alpha}(t).\]
which is easily integrated as:
\[v_{\alpha}(t)=e^{-\frac{2t}{\ta}} \Big( v_{\alpha}(0) + \int_0^t e^{\frac{2s}{\ta}}\lambda_{\alpha}^2(s)\,ds\Big). \]
$v_\alpha(t)$ is therefore  independent of the mean $\mu$. This implies that the equations on $\mu$ are a set of non-autonomous ordinary differential equations.

% From this property of the equations, we deduce that the equation on the mean is also an autonomous equation
% \[\dot{\mu_{\alpha}}=-\frac{\mu_{\alpha}}{\tau_{\alpha}} + \sum_{\beta=1}^P J_{\alpha\beta}f_{\beta}(\mu_{\beta},v_{\beta}(t))+I_{\alpha}(t)\]
% 
% \begin{equation}\label{eq:NMF}
%  \frac{d \mu_\alpha(t)}{dt}= 
% -\frac{\mu_\alpha(t)}{\tau_{\alpha}}+\sum_{\beta=1}^P J_{\alpha\beta}
% S_\beta\left(\mu_\beta(t)\right) +I_\alpha(t),
% \end{equation}
% %
These ordinary differential equations are similar to those of a single neuron. They differ in that the terms in the sigmoidal functions depend on the external noise levels $\lambda_{\alpha}(t)$. They read:
\[\dot{\mu}_{\alpha}=-\frac {\mu_{\alpha}}{\ta}  + \sum_{\beta=1}^P J_{\alpha\beta}\erf\left(\frac{g_{\beta}\, \mu_{\beta} + \gamma_{\beta}}{\sqrt{1+g_\beta^{2}e^{-2t/\tau_{\beta}} (v_{\beta}(0) + \int_0^t e^{2s/\tau_{\beta}}\lambda_{\beta}^2(s)\,ds)}}\right)+I_{\alpha}\\
\]
hence the slope  $g_{\beta}$ and the threshold  $\gamma_{\beta}$ are scaled by a time-varying coefficient which is always smaller than one.

We now focus on the stationary solutions when the noise parameter $\lambda$ does not depend upon time. In that case, the variance is equal to:
\[v_\alpha(t)=\ta \lambda_{\alpha}^2/2 + e^{-\frac{2t}{\ta}}(v_{\alpha}(0)-\ta \lambda_{\alpha}^2/2),\] 
and converges exponentially fast towards the constant value $\ta\lambda_{\alpha}^2/2$. Asymptotic regimes of the mean field equations are therefore Gaussian random variables with constant standard deviation. Their mean is solution of the equation:
\[\dot{\mu}_{\alpha}=-\frac {\mu_{\alpha}}{\ta} + \sum_{\beta=1}^P J_{\alpha\beta}\erf\left(\frac{g_{\beta}\, \mu_{\beta} + \gamma_{\beta}}{\sqrt{1+g_\beta^{2}\tau_{\beta}\lambda_{\beta}^2/2}}\right)+I_{\alpha}\quad \alpha=1,\cdots,P\\
\]
In other words, the presence of noise has the effect of modifying the slope $g_{\alpha}$ and the threshold $\gamma_{\alpha}$ of the sigmoidal $\erf$ function, but the type of the equations is the same as that of the equation of each individual neuron or cortical column, it is a rate equation. 

We observe that the larger the noise amplitude $\lambda$, the smaller the slope of the sigmoidal transform. Noise has the effect of smoothing the sigmoidal transform. This will have a  strong influence on the bifurcations of the solutions to the mean field equations and hence on the behaviors of the system. We demonstrate these effects for two simple choices of parameters  in one- and two-population networks. 

\subsection{The external  noise can destroy a pitchfork bifurcation}\label{ssec:PitchDestroy}
Let us start by considering the case of a one-population network. We drop the index $\alpha$  since no confusion is possible. We assume for simplicity that the threshold $\gamma$ of the sigmoid is null and that the time constant $\tau$ is equal to one. By doing so, we do not restrict the generality of the study, since  $\tau$ can be eliminated by rescaling the time and $\gamma$ can be absorbed into $I$ by a simple change of origin for $\mu$. The network equations read:
\[dV^i_t=\left(-V^i_t + \frac{J}{N} \sum_{j=1}^N \erf\left(g\, V^j_t \right) + I(t)\right)\, dt + \lambda dB^i_t\quad i=1,\cdots,N,\]
and we are interested in the limit in law of their solutions as the number of neurons $N$ tends to infinity. 

In order to analytically study the effect of the parameter $\lambda$, we set $I \equiv -\frac J 2$. In that case, and  in the absence of noise, the solution $V=0$ is a fixed point of the network equations. The following proposition characterizes their solutions in the deterministic and stochastic cases.

\begin{proposition}\label{pro:CancelPitch}
	In a non-stochastic finite-size network, the null solution is:
	\begin{itemize}
		\item stable if $J<0$ or if $J>0$ and $g<g_c:=\sqrt{2\pi}/J$,
		\item unstable for $J>0$ and  $g>g_c$.
		\item For $J>0$, the system undergoes a pitchfork bifurcation at $g=g_c$.
	\end{itemize} 
	In the mean field limit of the same stochastic network, the pitchfork bifurcation occurs for a new value of $g=g^*=\frac{\sqrt{2\pi}}{\sqrt{J^2-\pi\lambda^2}} > g_c$ if $J>0$ and $\lambda<J/\sqrt{\pi}$. Furthermore the null solution is:
	\begin{itemize}
		\item stable if:
		\renewcommand{\labelitemii}{$\star$}
		\begin{itemize}
			\item $J<0$ or
			\item $J>0$ and $\lambda > J/\sqrt{\pi}$ (large noise case) or
			\item $J>0$, $\lambda<J/\sqrt{\pi}$ and $g<g^*$,
		\end{itemize}
		\item unstable for $J>0$, $\lambda<J/\sqrt{\pi}$ and $g>g^*$, and 
		\item the system undergoes a pitchfork bifurcation at $g=g^*$ when $J>0$ and $\lambda<J/\sqrt{\pi}$. 
	\end{itemize}
\end{proposition}

This proposition is a bit surprising at first sight. Indeed, it says that noise can stabilize a fixed point which is unstable in the same non-stochastic system. Even more interesting is the fact that if the system is driven by a sufficiently large noisy input, the zero solution will always stabilize. It is known, see, e.g.,~\cite{mao:08}, that noise can stabilize the fixed points of a deterministic system of dimension greater than or equal to $2$. The present observation extends these results to a one-dimensional case, in a more complicated setting because of the particular, non-standard, form of the mean field equations. Also note that this proposition provides a precise quantification of the value of the parameter that destabilizes the  fixed point. This is a stochastic bifurcation of the mean field equation (a P-bifurcation --P for phenomenological-- in the sense of~\cite{arnold:98}). This estimation will be used as a yardstick for the evaluation of the behavior of the solutions to the network equations in section \ref{sec:Network}. 

\begin{proof}
	We start by studying the finite-size deterministic system. In the absence of noise, it is obvious because of our assumptions that the solution $V^i=0$ for all $i\in\{1,\ldots,N\}$ is a fixed point of the network equations. At this point, the Jacobian matrix reads $-Id_N+\frac{J}{N}\,\frac{g}{\sqrt{2\pi}}\,\mathbbm{1}_N$, where $Id_N$ is the $N \times N$ identity matrix and $\mathbbm{1}_N$ is the $N\times N$ matrix with all elements equal to one. The matrix $\mathbbm{1}_N$ is diagonalizable, all its eigenvalues are equal to zero except one which is equal to $N$. Hence, all eigenvalues of the Jacobian matrix are equal to $-1$, except one which is equal to $\frac{J\,g}{\sqrt{2\pi}}-1$. The solution where all $V^i$ are equal to zero in the deterministic system is therefore stable if and only if $g\,J<\sqrt{2\pi}$. The eigenvalue corresponding to the destabilization  corresponds to the eigenvector $\overrightarrow{1}$ whose components are all equal to $1$. Interestingly, this vector does not depend on the parameters, and therefore it is easy to check that at the point $g=g_c$ the system loses stability through a pitchfork bifurcation. Indeed, because of the symmetry of the $\erf$ function, the second derivative of the vector field projected on this vector vanishes, while the third derivative does not (it is equal to $-(1+g^2)$). 

Considering now the stochastic mean field limit, the stationary mean firing rate in that case is solution of the equation:
\[\dot{\mu}=-\mu + J\erf\left(\frac{g\, \mu}{\sqrt{1+g^{2}\lambda^2/2}}\right)+I\]
	Here again, the null firing rate point $\mu=0$ is a fixed point of the mean field equations, and it is stable if and only if $-1+J\,\frac{g}{\sqrt{2\pi(1+g^{2}\lambda^2/2)}}<0$. The remaining of the proposition  readily follows from the fact that the stability changes at $g=g^*$ where $J\,\frac{g^*}{\sqrt{2\pi(1+{g^*}^{2}\lambda^2/2)}}=1$.
\end{proof}

Note that the results in this proposition only depend on $\lambda$ and its effect on the slope of the sigmoid. 
 It is a general phenomenon that goes beyond the example in this section: increasing $\lambda$  decreases the slope of the sigmoidal transform and the threshold.  In section~\ref{sec:Network} we will see that this phenomenon can be observed at the network level, and a good agreement will be found between the finite-size network behavior and the predictions obtained from the mean field limit. 

We now turn to an example in a two-dimensional network, where the presence of oscillations will be modulated by the noise levels.  

\subsection{The external noise can destroy oscillations}\label{ssec:NoiseOscill}
The same phenomenon of nonlinear interaction between the noise intensity and the sigmoid function can lead, in higher dimensions, to more complex phenomena such as the disappearance or appearance of oscillations. In order to study phenomena of this type, we instantiate a simple two-populations network model in which, similarly to the one-dimensional case, all the calculations can be performed analytically. The network we consider consists of an excitatory population, labeled $1$, and an inhibitory population, labeled $2$. Both populations are composed of the same number $N/2$ of neurons ($N$ is assumed in all the subsection to be even), and have the same parameters $\tau_{1}=\tau_2=\tau$, $g_{1}=g_2=g$ and $\lambda_{1}=\lambda_2=\lambda$. We choose for simplicity the following connectivity matrix:
\[
M=J\, \frac {2}{N} \left ( \begin{array}{cc}
	1 & -1\\
	1 & 1
\end{array}\right),
\]
and we assume that the inputs are set to $I_1=0$ and $I_2=-J$. The zero solution where all neurons have a zero voltage is a fixed point of the equations whatever the number of neurons $N$ in each population. We have the following result:

\begin{proposition}\label{pro:CyclesNoise}
	In the deterministic finite-size network, the null solution is:
	\begin{itemize}
		\item stable if $J<0$ or if $J>0$ and $g<g_c:=\sqrt{2\pi}/J$,
        \item unstable for $J>0$ and $g>g_c$ and the solutions are oscillating on a periodic orbit.
		\item For $J>0$ the system undergoes a supercritical Hopf bifurcation at $g=g_c$.
	\end{itemize} 
	In the mean field limit of  the same stochastic network, the Hopf bifurcation occurs for a new value of the slope parameter  $g=g^*=\frac{\sqrt{2\pi}}{\sqrt{J^2-\pi\lambda^2}} > g_c$. Furthermore the null solution is:
	\begin{itemize}
		\item stable if:
		\renewcommand{\labelitemii}{$\star$}
		\begin{itemize}
			\item $J<0$ or
			\item $J>0$ and $\lambda > J/\sqrt{\pi}$ (``large noise case'') or
			\item $J>0$, $\lambda<J/\sqrt{\pi}$ and $g<g^*$,
		\end{itemize}
		\item unstable for $J>0$, $\lambda<J/\sqrt{\pi}$ and $g>g^*$, and the system features a stable periodic orbit.
		\item The system undergoes a supercritical Hopf bifurcation at $g=g^*$ when $J>0$ and $\lambda<J/\sqrt{\pi}$. 
	\end{itemize}
\end{proposition}

Note that proposition \ref{pro:CyclesNoise} is quite similar to proposition~\ref{pro:CancelPitch}, the qualitative difference being that the system is oscillating. The proof is closely related and is presented in less details.

\begin{proof}
In the deterministic network model, the Jacobian matrix at the null equilibrium can be written as 
\[A=-Id_{N} + \frac{g}{\sqrt{2\pi}}\,M \otimes \mathbbm{1}_{N/2}\] 
where $\otimes$ denotes the Kronecker product (see e.g.~\cite{neudecker:69,brewer:78}), i.e. the Jacobian matrix is built from  $N/2$ blocs of size $2\times2$ and each of these blocks is a copy of $\frac{g}{\sqrt{2\pi}}\,M$. The eigenvalues of a Kronecker product of two matrices are all possible pairwise products of the eigenvalues of the matrices. Since the eigenvalues of $M$ are equal to $\frac{2\,J}{N}\,(1\pm \mathbf{i})$  where $\mathbf{i}^2=-1$, and as noted previously, the eigenvalues of $\mathbbm{1}_{N/2}$ are $0$ with multiplicity $N/2-1$ and $N/2$ with multiplicity $1$, we conclude that the Jacobian matrix $A$ has $N-2$ eigenvalues equal to $-1$, and two eigenvalues  equal to $-1+g\,J/\sqrt{2\pi}(1\pm \mathbf{i})$. The null equilibrium in this deterministic system is therefore stable if and only if the real parts of all eigenvalues are smaller than $0$, viz. $gJ<\sqrt{2\pi}$. Therefore, for a fixed $J$, the system has a bifurcation at $g_c=\sqrt{2\pi}/J$. The analysis of the eigenvectors allows to check the genericity and transversality conditions of the Hopf bifurcation (see e.g.~\cite{guckenheimer-holmes:83}) in a very similar fashion to the proof of proposition~\ref{pro:CancelPitch}. 

In the mean field model, the same analysis applies and, as in the one-dimensional case, the bifurcation point is shifted to $g^*$ when this value is well-defined, which concludes the proof of the proposition~\ref{pro:CyclesNoise}. 
\end{proof}

We have therefore shown that noise can destroy the oscillations of the network. The results of propositions \ref{pro:CancelPitch} and \ref{pro:CyclesNoise} are summarized in figure \ref{fig:pitchfork-hopf}.
\begin{figure}[!h]
\centerline{
\includegraphics[width=0.75\textwidth]{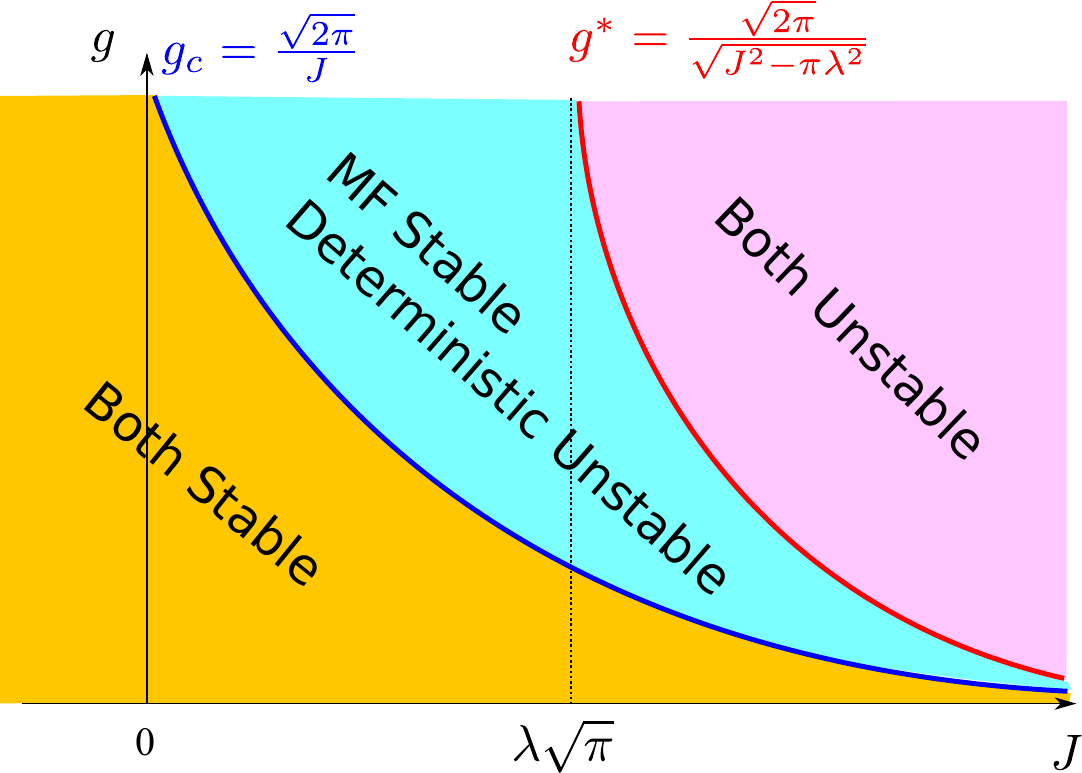}
}
\caption{Summary of the results in propositions \ref{pro:CancelPitch} and \ref{pro:CyclesNoise}, see text. The additive noise parameter $\lambda$ smoothly modifies the pitchfork or the Hopf bifurcation curve in the  $(g,J)$ plane. For $\lambda$ large enough, the null solution of the mean field equation is always stabilized whatever $g$.MF: mean field limit, Deterministic: finite-size deterministic network. The blue (respectively red) curve is one branch of the hyperbola of equation $gJ=\sqrt{2\pi}$ (respectively $g\sqrt{J^2-\pi \lambda^2}=\sqrt{2\pi}$).}
\label{fig:pitchfork-hopf}
\end{figure}

An even more interesting phenomenon is that noise can also produce regular cycles in the mean part of the solution of the mean field equations, for parameters such that the deterministic system presents a stable equilibrium. This is the subject of the following section.

\subsection{The external noise can induce oscillations}\label{sec:TwoPopsLambda}
\review{In order to uncover further effects of the noise on the dynamics, we now turn to the numerical study of a two-populations network including excitation and inhibition. The time constant $\tau$, sigmoidal transforms $S$, noise intensity $\lambda$ and the initial condition on the variance are chosen identical for both population. Under these hypotheses, the variances of the two populations are identical and denoted by $v(t)$. We further assume that $S(x)=\erf(g\,x)$, and hence the mean field nonlinear function $f(\mu,v)$ is given by lemma~\ref{lemma:ErfSigmoids}. The connectivity matrix is set to $J=\frac{1}{N}\left ( \begin{array}{cc}
		15 & -12\\
		16 & -5
	\end{array}
	\right)$.
% where $j$ accounts for the total connectivity strength of the network.
% 
% This system was studied in a different context in~\cite{touboul-ermentrout:11} where it displayed an interesting, non-trivial dynamics.  
The input currents $I_1$ and $I_2$ are considered constant. The mean field equations in that case read:
\begin{equation*}
\begin{cases}
	\dot{\mu}_{1}=-\frac{\mu_{1}}{\tau} +  J_{11}f(\mu_{1},v)+J_{12}f(\mu_{2},v)+I_{1}\\
        \dot{\mu}_{2}=-\frac{\mu_{2}}{\tau} +J_{21}f(\mu_{1},v)+J_{22}f(\mu_{2},v)+I_{2}\\
	\dot{v}=-2\frac{v}{\tau} +\lambda ^2
\end{cases}
\end{equation*}
}
% We study the influence of the noise on the behaviors as an input parameter is varied. More precisely, we want to define classes of behaviors corresponding to different bifurcation diagrams, as the deterministic external input parameter $I_1$ is left free while the other parameter $I_2$ is kept constant and equal to $-3$ (qualitative results turn out to change smoothly when $I_2$ is also allowed to vary). We proceed with the results of the numerical analysis of the bifurcations of the system, using XPPAut~\cite{ermentrout:02} for codimension one diagrams and MatCont package~\cite{dhooge-govaerts-etal:03,dhooge-govaerts-etal:03b} for codimension two diagrams. 

% Six qualitative different in the dynamics of the neuronal population in response to different inputs appear when varying the noise level $\lambda$.

\review{The codimension two bifurcation diagram of the system for $I_2=-3$ is displayed in Figure~\ref{fig:Codim2Lambda} (qualitative results turn out to change smoothly when $I_2$ is also allowed to vary). It features two cusps (CP) and one Bogdanov-Takens  (BT) bifurcations. In addition to these local bifurcations, we observe that the Hopf bifurcation manifold (shown in pink in Figure~\ref{fig:Codim2Lambda}) and the saddle-homoclinic bifurcation curve (green line) present a turning point, i.e. change monotony as a function of $\lambda$. }

\begin{figure}[htbp]
	\centering
		\includegraphics[width=\textwidth]{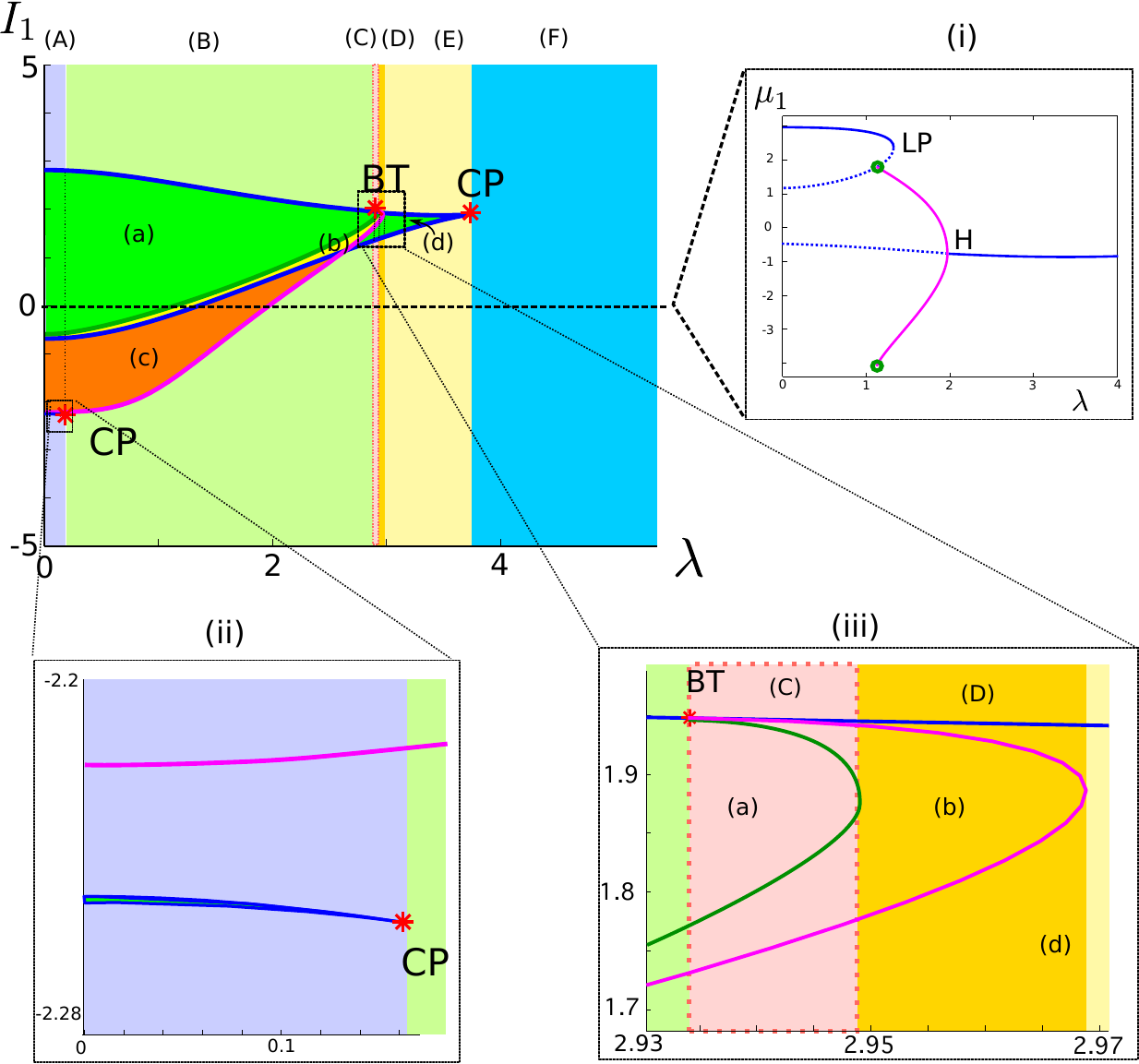}
	\caption{\review{Codimension two bifurcation diagram (upper left) and zooms (subfigures (ii) and (iii)) for the mean field equations as $I_1$ and $\lambda$ are varied. We distinguish, apart from the trivial regime with a single fixed point, three dynamical regimes labeled (a), (b) and (c) (see text) and $6$ ranges of $\lambda$, labeled (A) through (F). Blue: saddle-node bifurcations, pink: Hopf bifurcations, green: saddle homoclinic bifurcations, BT: Bogdanov Takens bifurcation, CP: cusp. Individual behaviors in each zone are summarized in appendix~\ref{sec:BifLambda}. (i): Codimension 1 bifurcation diagram for $I_1=0$ as a function of $\lambda$: saddle-node (LP), a Hopf (H) and a saddle homoclinic bifurcation (green circles). We observe three main different noise regimes: a high-state equilibrium regime, a periodic regime and a low-state equilibrium regime. A small interval of values of $\lambda$ corresponds to the co-existence of cycles and a fixed point close to the saddle-homoclinic orbit. Diagrams obtained with XPPAut~\cite{ermentrout:02} and MatCont package~\cite{dhooge-govaerts-etal:03,dhooge-govaerts-etal:03b}. }}
	\label{fig:Codim2Lambda}
\end{figure}

\review{The diagram can be decomposed into $4$ different regions depending on the dynamical features (number and stability of fixed points or cycles): the ``trivial'' zone where the system features a unique stable fixed point (not colored), a zone with 2 unstable and 1 stable fixed point (green zone (a)) separated by the saddle-homoclinic bifurcation curve from region (b) (yellow) where an additional stable cycle exists. Zone (c) (orange) features a stable cycle and an unstable fixed point, and zone (d) (green) features 2 stable and 1 unstable fixed points. }

\review{Let us for instance fix $I_1=0$. As $\lambda$ is increased, several noise-induced transitions occur leading the system successively in zone (a), (b), (c) and the trivial zone (see codimension one bifurcation diagram in Figure~\ref{fig:Codim2Lambda} (i)). In details, for small noise levels the system features a unique stable fixed point (zone (a)). A family of large amplitude and small frequency periodic orbits appears from the saddle-homoclinic bifurcation yielding a bistable regime (zone (b)) before the stable fixed point disappears through a saddle-node bifurcation (zone (c)). The amplitude of these cycles progressively decreases and their frequency progressively increases as the noise intensity is increased, and they eventually disappear through a supercritical Hopf bifurcation leading to the trivial behavior with a single fixed point. We emphasize here the fact that the sudden appearance of large amplitude slow oscillations can be compared to epileptic spikes, which are characterized by the presence of collective oscillations of large amplitude and small frequency suddenly appearing in a population of neurons (see~\cite{touboul-faugeras-wendling:10}). This comparison turns out to be relevant from the microscopic viewpoint: network simulations of section~\ref{sec:Network} will indeed show a sudden synchronization of all neurons at this transition. }

\review{The diagram can also be decomposed into six different noise levels intervals (labels (A) through (F) in Figure~\ref{fig:Codim2Lambda}) corresponding to qualitatively different codimension 1 bifurcation diagrams as $I_1$ is varied (the six corresponding bifurcation diagrams are presented in appendix~\ref{sec:BifLambda}). The presence of these different zones illustrate how the noise influences the response of the neural assembly to external inputs. For instance, for  $\lambda$ large enough, no cycles exist whatever $I_1$ (zones (E-F)), whereas for $\lambda$ small enough (zones A-D), cycles always exist for some values of the input. Such partitions may provide an experimental design for evaluating a noise level range as a function of the observed dynamics when varying the input to the excitatory population for instance.}

\section{Back to the network dynamics}\label{sec:Network}
Thus far, we studied the dynamics of the mean field equations representing regimes of the network dynamics in the limit where the number of neurons is infinite. We now compare the regimes identified in this analysis with simulations of the finite-size stochastic network. We are particularly looking for potential finite-size effects, namely qualitative differences between the solutions to the network and the mean field equations. This will provide us with information about the accuracy of an approximation of the network dynamics by the mean field model, as function of the size of the network. 

\subsection{Numerical Simulations}
Numerical simulations of the network stochastic differential equations~\eqref{eq:Network} are performed using the usual Euler-Maruyama algorithm (see e.g. \cite{maruyama:55,mao:08}) with fixed time step (less than $0.01$) over an interval $[0,T]$. In order to observe oscillations, we choose $T$ between $50$ and $70$. The simulations are performed with Matlab\textregistered, using a vectorized implementation that has the advantage to be very efficient even for large networks. The computation time stays below $1$s for networks up to $2\,000$ neurons, and appears to grow linearly with the size of the network once the cache memory saturates (see Figure~\ref{fig:TempsCalcul}). For instance, for $T=20$, $dt=0.01$, the simulation of a $2\,000$ neurons network takes $0.89$s, and for $525\,000$ neurons, $600$s on a HP Z800 with 8 Intel Xeon CPU E5520 @ 2.27 GHz 17.4 Go RAM. The main limitation preventing the simulation of very large networks is the amount of memory required for the storage of the trajectories of all neurons.
\begin{figure}
	\centering
		\includegraphics[width=.4\textwidth]{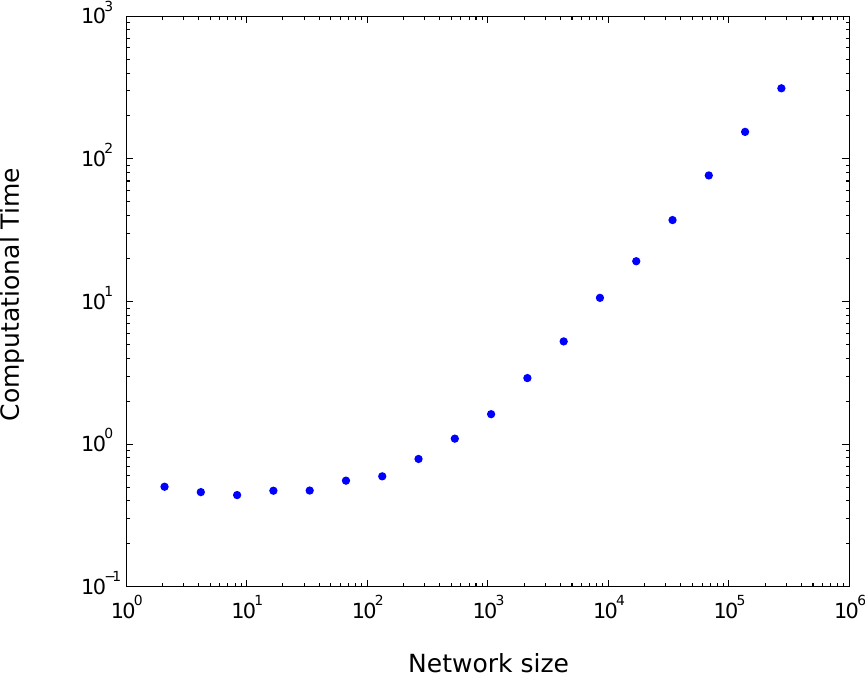}
	\caption{Computation time for the simulation of the stochastic network in logarithmic scale as a function of the network size.}
	\label{fig:TempsCalcul}
\end{figure}

An important property arising from theorem~\ref{thm:propagationchaos} is that asymptotically, neurons behave independently and have the same probability distribution. In our numerical simulations, we will make use of this asymptotic independence and, in order to evaluate an empirical mean of the process related to a given neuron in population $\alpha$, will compute both the empirical mean over all neurons in that population and a mean over different independent realization of the process. This method allows to reduce sensitively the number of independent simulations in order to obtain a given precision in the empirical mean evaluation.

\subsection{A one population case}
We start by addressing the case discussed in section~\ref{ssec:PitchDestroy} where we showed analytically that the loss of stability of the null fixed point as the slope of the sigmoid was varied depended on the noise parameter $\lambda$. We now investigate numerically the stability of the $0$ fixed point of the network equations. In order to check for the presence of a pitchfork bifurcation, we compute, for each value of the noise and for each value of the slope of the sigmoid, an estimated value of the mean of the membrane potential. This estimate is calculated by averaging out over $500$ independent realizations the empirical mean of the membrane potentials of all neurons in the network at the final time. We display the average value then compare these simulations with those of the mean field equations stopped at the same time as the network. We observe that both are very similar and show some differences with the bifurcation diagram that corresponds to the asymptotic regimes.

The results of the simulations, where we have also varied $N$, are shown in Figure~\ref{fig:bifdelayed} and reveal two interesting features. First, because we simulate over a finite time, we tend to smooth the pitchfork bifurcation: this is perceptible for both the network and the mean field equations. Second, we observe that the loss of stability of the zero fixed point arises at the value of $\lambda$ predicted by the analysis of the mean field equations for networks as small as $50$ neurons. The value reached by the simulations of the network is very close to that related to the mean field equation as soon as $N$ becomes greater than $250$. 

\begin{figure}[htbp]
	\begin{center}
		\subfigure[Network Simulations vs mean field simulations, different $\lambda$, $N=50\,000$]{\includegraphics[width=.45\textwidth]{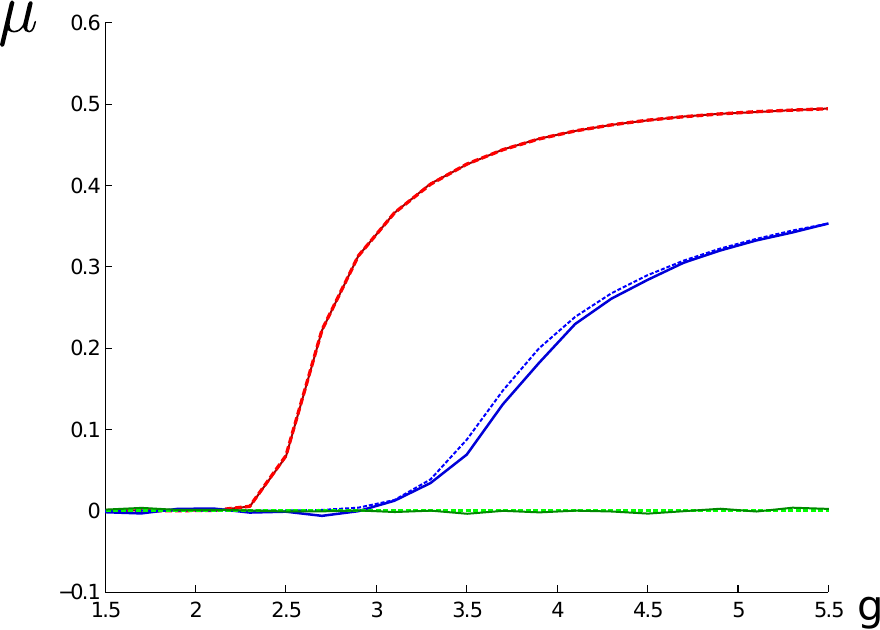}}\quad
		\subfigure[Network simulations vs mean field simulations, different $N$, $\lambda=0.4$]{\includegraphics[width=.45\textwidth]{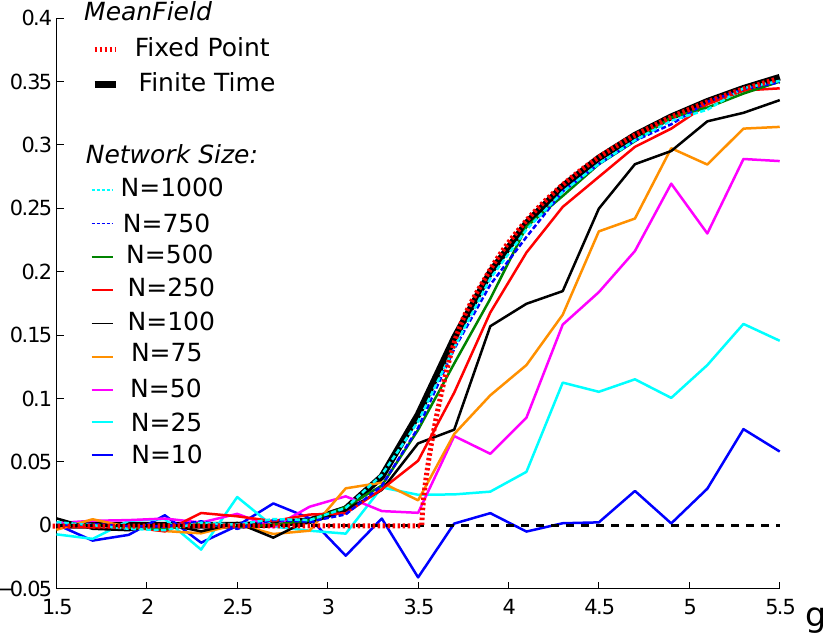}}
	\end{center}
	\caption{Comparison of the pitchfork bifurcations with respect to the slope parameter $g$ for the network and the mean field equations, $T=40$, $dt=0.001$, number of sample paths: $100$, initial condition $V^0=0.5$ (hence we only see the positive part of the pitchfork, symmetrical solutions are found for negative initial conditions, and are not plotted for legibility). (a): $50\,000$ neurons. Continuous curves correspond to network simulations, dashed curves to mean field simulations.  When $\lambda$ increases, as predicted by proposition~\ref{pro:CancelPitch}, we observe that the value of the parameter $g$ related to the pitchfork bifurcation increases as well, until the pitchfork  disappears: red: $\lambda=0$, blue: $\lambda=0.4$, green: $\lambda=0.8>\sqrt{2\pi}/J\sim 0.56$. (b): $\lambda=0.4$. The solution to the  mean field equation undergoes a pitchfork bifurcation at $g=3.55$. Large dotted red: theoretical pitchfork bifurcation. Large black: endpoint of mean field simulation at time $T=40$. The other colored curves show the results of the network simulation for different values of the size of the network $N$. The $0$ solution, which loses stability, is displayed in thin dashed black. We see that as $N$ increases, the mean field equation describes accurately the network activity. For $N\geq 50$ (red, green, dotted blue and dotted cyan curves) the bifurcation diagram is quite close to the one predicted by the mean field analysis.}
	\label{fig:bifdelayed}	
\end{figure}

\subsection{Two populations case and oscillations}
We now investigate the case shown in Figure~\ref{fig:Codim2Lambda} (i) where cycles are created (through homoclinic bifurcation) or destroyed (through Hopf bifurcation) as the additive noise intensity parameter $\lambda$ is increased. 

Looking at  Figure~\ref{fig:Codim2Lambda}(i), we observe that for $\lambda \in [1.12, 1.33]$, stable periodic orbits coexist with stable fixed points in the mean field system. For smaller values of $\lambda$, the mean field system features a unique stable fixed point, while for $\lambda \in [1.33, 1.97]$, it features a unique stable limit cycle, and for $\lambda>1.97$, the dynamics is reduced to a unique attractive fixed point. Numerical simulations confirm this analysis. Let us for instance illustrate the fact that the network features the bistable regime, the most complex phenomenon. Figure~\ref{fig:Bistable} shows simulations of a network composed of $5\,000$ neurons in each population (time step $dt=5\cdot 10^{-3}$, total time $T=50$). Depending on the mean and on the standard deviation of the initial condition, we observe that the network either converges to the mean field fixed point or the periodic orbit.  Both mean field equations show very close behaviors.  
\begin{figure}[htbp]
	\centering
		\subfigure[$\lambda=1.2$. \textbf{Oscillatory regime}. Statistics of the network compared to the mean field.]{\includegraphics[width=.45\textwidth]{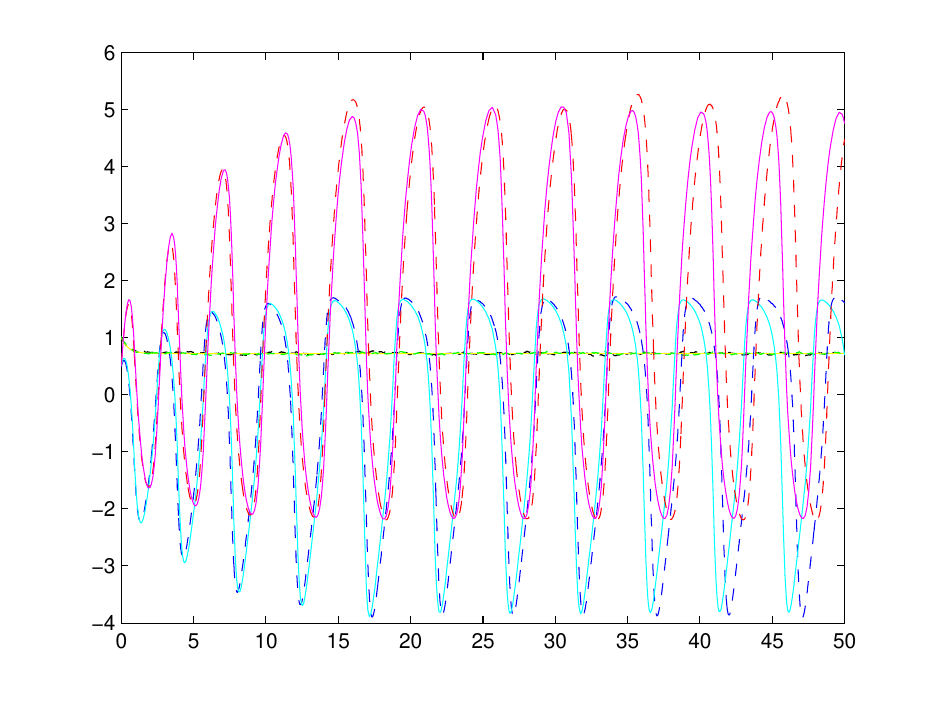}}\quad 
		\subfigure[$\lambda=1.2$. \textbf{Fixed-point regime} Statistics of the network compared to the mean field.]{\includegraphics[width=.45\textwidth]{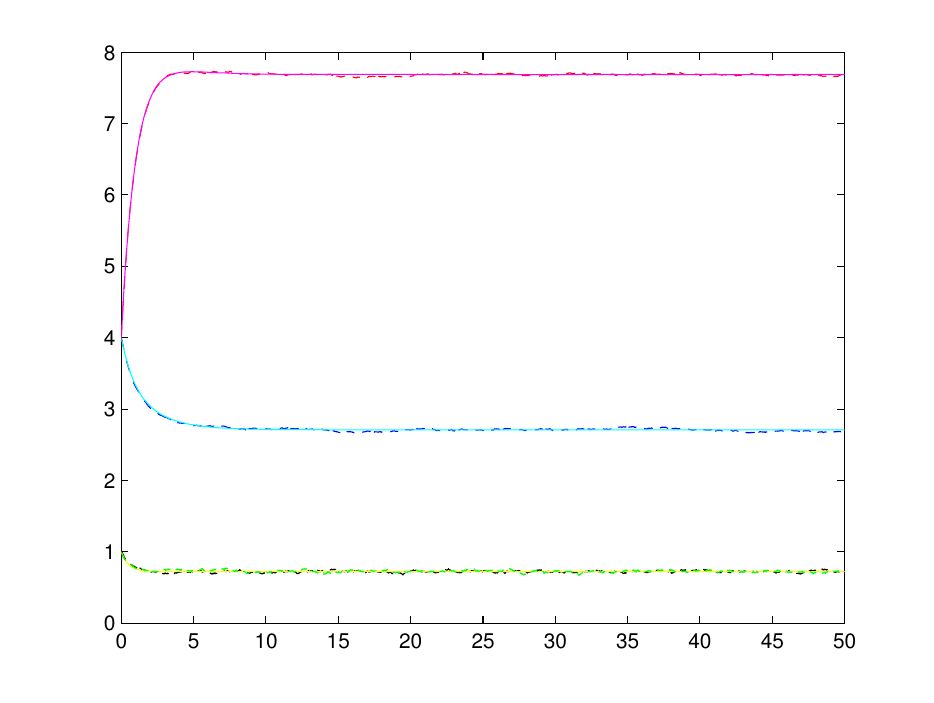}}               
                \caption{Featuring bistability. In both cases $\lambda=1.2$. The initial conditions for the mean field equation are chosen in agreement with the initial conditions of the network. The initial value of the membrane potential of each individual neuron in the network is drawn independently from a Gaussian distribution of variance 1 whose mean  depends on the population: (a)  mean $0.5$. (b)  mean $4$. Cyan (resp. magenta) curves: value of the mean variable of the mean field solution for population $1$ (resp. 2). Dashed blue (resp. red) curves: empirical mean of population $1$ (resp. $2$). Yellow: value of the variance of the mean field solution. Dashed black (resp. green): empirical variance of population $1$ (resp. $2$).}
          \label{fig:Bistable}
\end{figure}
\\
In the fixed point regime corresponding to small values of $\lambda$ we observe that the membrane potential of every neuron randomly varies around the value corresponding to the fixed points of the mean field equation (see Figure~\ref{fig:NetworkLambda}, cases (a) and (b))), with a standard deviation that converges toward the constant value $\lambda^2/2$ as predicted by the mean field equations. The empirical mean and standard deviation of the voltages in the network show a very good agreement with the related mean field variables. For larger values of $\lambda$ corresponding to the oscillatory regime (Figure~\ref{fig:NetworkLambda}, cases (c) and (d)), all neurons oscillate in phase. These synchronized oscillations yield a coherent global oscillation of the network activity. The statistics of the network are again in good agreement with the mean field solution. The standard deviation converges towards the constant solution of the mean field equation. This is visible at the level of individual trajectories, that shape a ``tube'' of solutions around the periodic mean field solution, whose size increases with $\lambda$. The empirical means accurately match the regular oscillations of the solution of the mean field equation. A progressive phase shift is observed, likely to be related with the time step $dt$ involved in the simulation. Note that the phase does not depend on the realization. Indeed, according to theorem~\ref{thm:propagationchaos}, the solution of the mean field equations only depends on the mean and the standard deviation of the Gaussian initial condition, which therefore governs the phase of the oscillations on the limit cycle (see Figure~\ref{fig:phase}).

In the fixed point regime related to large values of $\lambda$, very noisy trajectories are obtained because of the levels of noise involved (see Figure~\ref{fig:NetworkLambda}, cases (e) and (f)). Though the individual neurons show very fluctuating trajectories, the empirical mean averaged out over all neurons in the network fits closely the mean field fixed point solution.

\begin{figure}[htbp]
	\centering
		\subfigure[$\lambda=0.6$. \textbf{Fixed-point regime}. Individual trajectories vs mean field.]{\includegraphics[width=.45\textwidth]{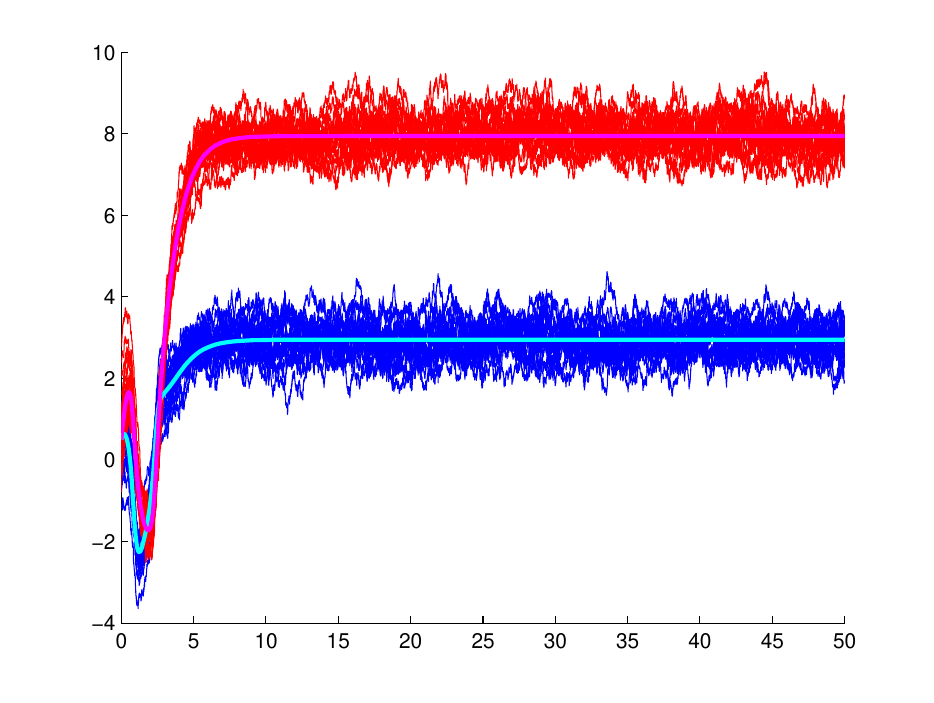}}\quad 
		\subfigure[$\lambda=0.6$. \textbf{Fixed-point regime} Empirical network statistics vs mean field.]{\includegraphics[width=.45\textwidth]{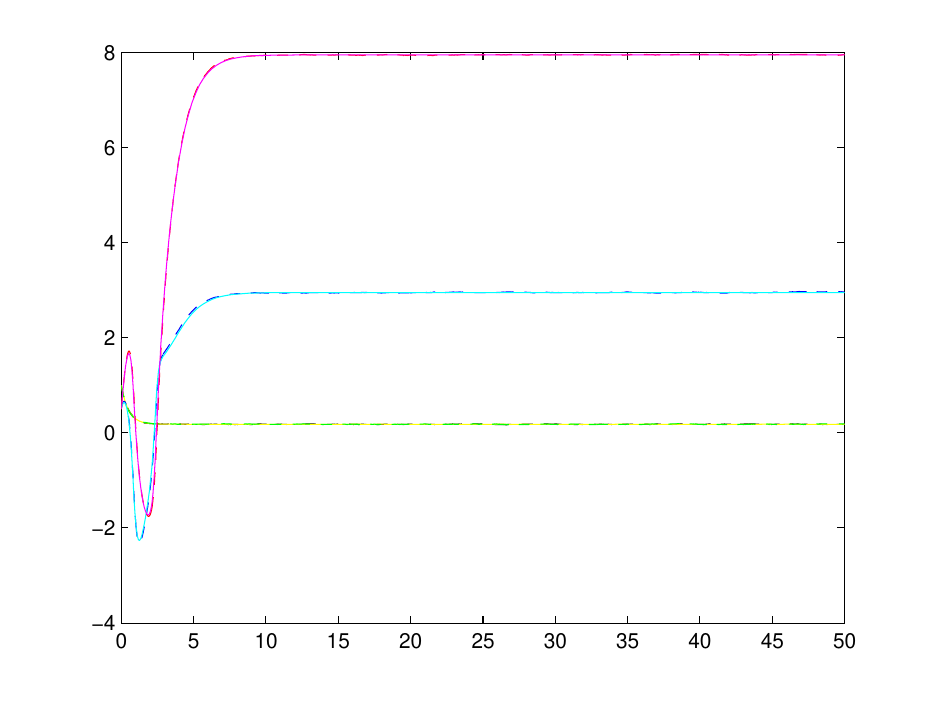}}
                \subfigure[$\lambda=1.2$. \textbf{Oscillatory regime}. Individual trajectories vs mean field.]{\includegraphics[width=.45\textwidth]{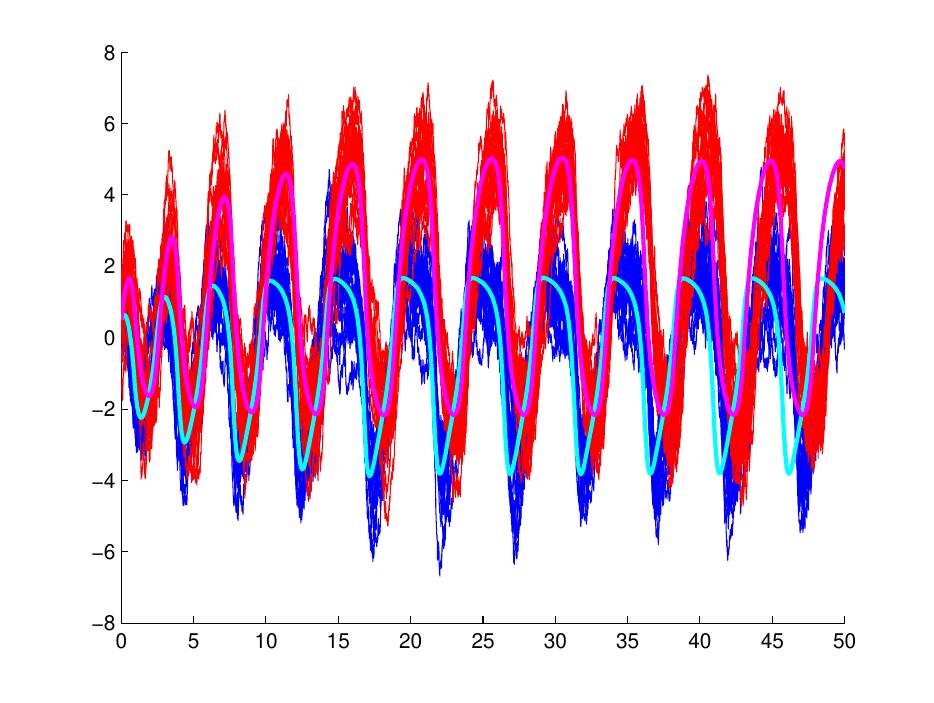}}\quad 
		\subfigure[$\lambda=1.2$. \textbf{Oscillatory regime} Empirical network statistics vs mean field.]{\includegraphics[width=.45\textwidth]{statlambda1_2}}
                \subfigure[$\lambda=2.5$. \textbf{Noisy fixed point regime}. Individual trajectories vs mean field.]{\includegraphics[width=.45\textwidth]{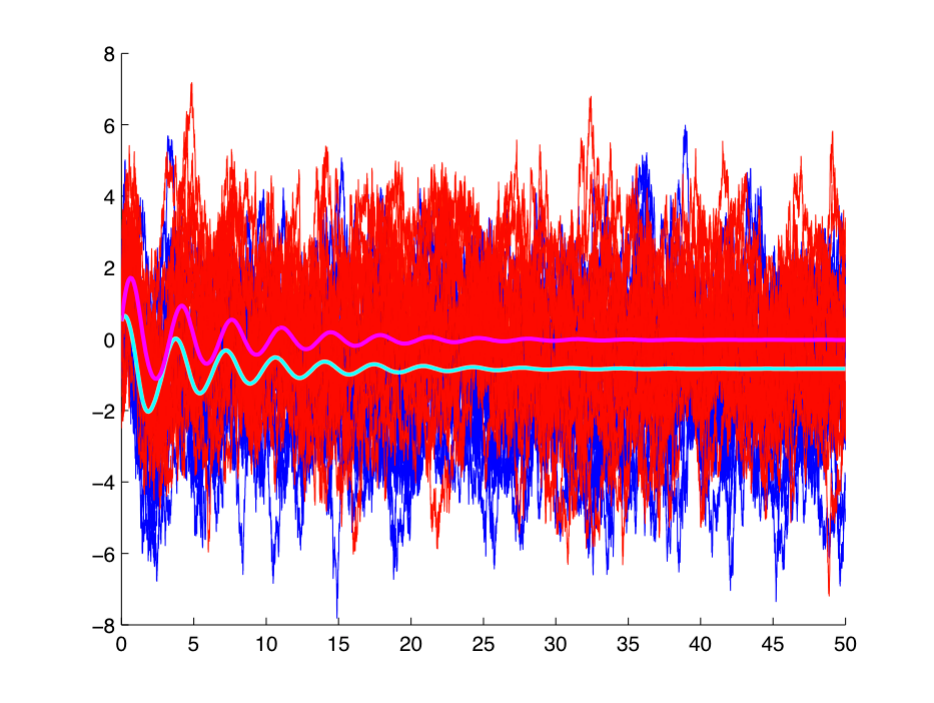}}\quad 
		\subfigure[$\lambda=2.5$. \textbf{Noisy fixed point regime} Empirical network statistics vs mean field.]{\includegraphics[width=.45\textwidth]{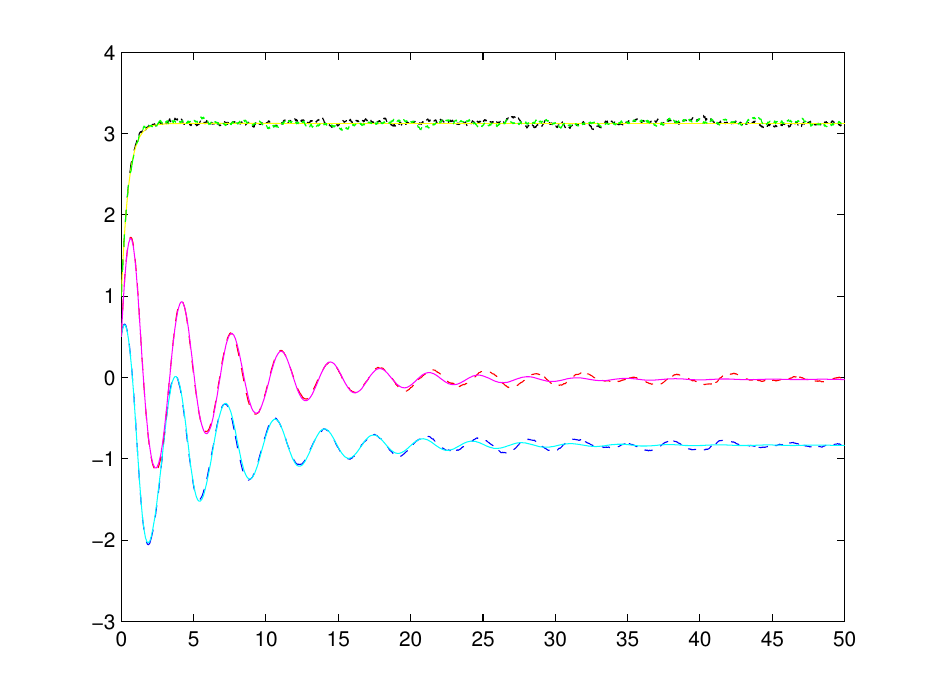}}
                \caption{Solution of the network dynamics for different values of the noise parameter $\lambda$ compared to the mean field solution. Simulations are run for $10\,000$ neurons, $5\,000$ in each population. (a), (c), (e): $40$ individual trajectories of the membrane potentials of $40$ neurons arbitrarily chosen in the network ($20$ in each population) compared to the solution of the mean field equations. Blue: population $1$ (excitatory). Red: population $2$ (inhibitory). Cyan (resp. magenta): mean of the mean field solution for population $1$ (resp. 2). (b), (d), (f): Empirical statistics of the network compared to the mean field. Cyan (resp. magenta): mean of the mean field solution for population $1$ (resp. 2). Yellow: variance of the mean field solution. Dashed blue (resp. red): empirical mean of population $1$ (resp. $2$). Dashed black (resp. green): empirical variance of population $1$ (resp. $2$). \review{For $\lambda=2.5$, due to the amplitude of noise, the statistics were averaged over $10$ realizations of the process.}}
          \label{fig:NetworkLambda}
\end{figure}

\begin{figure}[htbp]
	\centering
                 \includegraphics[width=0.60\textwidth]{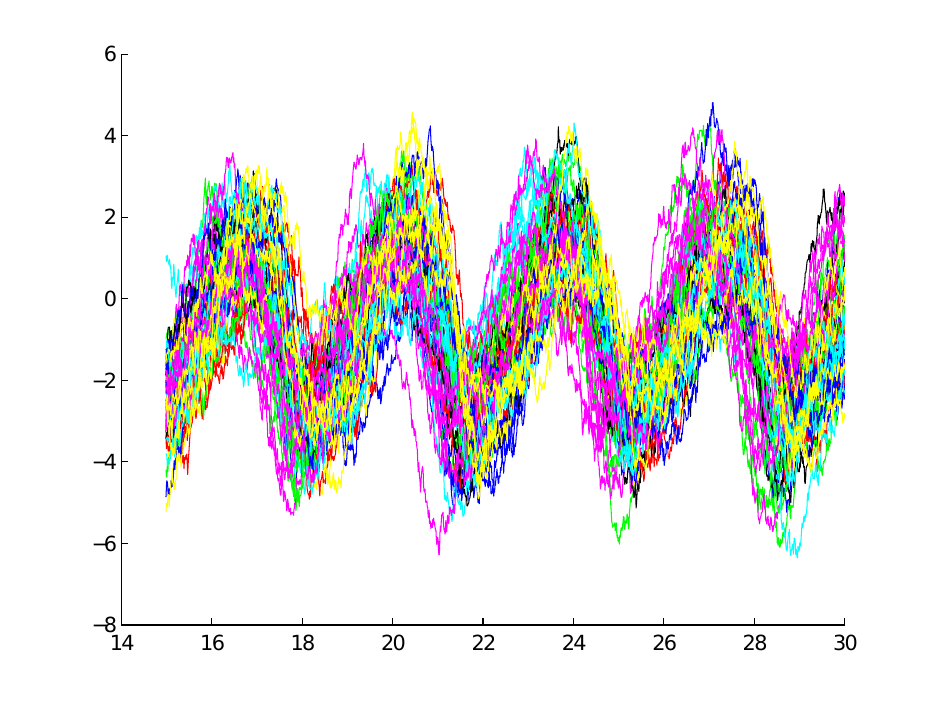}
	\caption{Different realizations of the stochastic network dynamics: the membrane potentials of $5$ neurons among $5\,000$ of population 1 are plotted for $12$ different realizations represented in different colors. All neurons oscillate in phase, and this phase does not depend on the realization.}
 \label{fig:phase}
\end{figure}

Eventually, we study the switching between a fixed-point regime and an oscillatory regime by extensively simulating the $10\,000$ neurons network for different values of $\lambda$ and computing the Fourier transform of the empirical mean (see Figure~\ref{fig:Lambda3d}). The three-dimensional plots show that the appearance and disappearance of oscillations occur for the same values of the parameter $\lambda$ as in the mean field limit, and the route to oscillations is similar: at the homoclinic bifurcation in the mean field system, arbitrarily small frequencies are present, this is also the case for the finite-size network. At the value of $\lambda$ related to the Hopf bifurcation, the system suddenly switches from a non-zero frequency to a zero frequency in a form that is very similar to the network case. Therefore we conclude that the mean field equations accurately reproduce the network dynamics for networks as small as $10\,000$ neurons, and hence provide a good model, simple to study, for networks of the scale of typical cortical columns. As a side remark, we note that at a homoclinic bifurcation of the mean field system, very small frequencies appear and a precise description of the spectrum of the network activity would require very large simulation times to uncover precisely the spectrum at this point, even more so since the large standard deviation of the process disturbs the synchronization. In a forthcoming study focusing on mean field equations arising in a similar system including dynamically fluctuating synaptic coefficients, an interesting additional phenomenon will appear: in that case, the standard deviation variable will be coupled to the mean, and this coupling will result in sharpening the synchronization.
\begin{figure}[htbp]
	\centering
                \includegraphics[width=\textwidth]{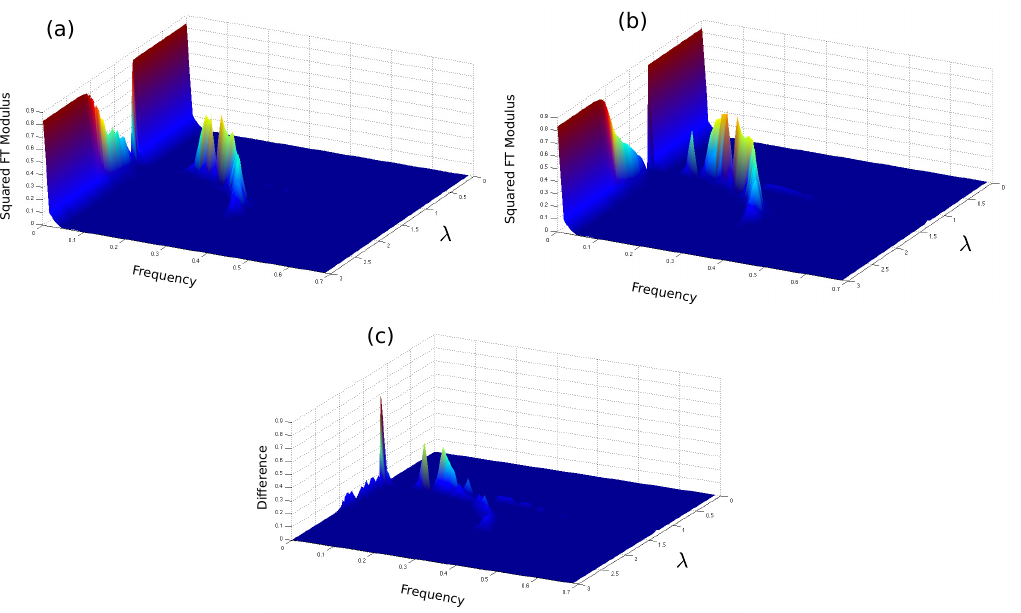}
	\caption{Squared moduli of the Fourier transforms of (a) the empirical mean for simulations of the network  and (b) of the mean variable of the solution to the mean field equations as functions of the frequency (Hz) and the noise parameter $\lambda$. We observe that oscillations  appear in the network for the same value of $\lambda$ as in the mean field equations (Figure~\ref{fig:Codim2Lambda}), first through what appears to be a  homoclinic bifurcation (arbitrary small frequencies) and also disappear for the same value of  $\lambda$ through what seems to be a Hopf bifurcation (discontinuity in the power spectrum). (c) Magnitude of the difference between the two diagrams: we note that the frequency distribution reaches its maxima for these same values of $\lambda$, and the main differences are observed, as expected, around the putative homoclinic bifurcation point.}
	\label{fig:Lambda3d}
\end{figure}

\medskip

We conclude this section by discussing heuristic arguments explaining the observed regular oscillations. Let us start by stating that this phenomenon is a pure collective effect: indeed, two-neurons networks (one per population) do not present such regular oscillations as noise is varied. We observe that individual trajectories of the membrane potential of a 2-neurons networks for small noise levels stay close to the deterministic fixed point. However, when noise is increased, the system starts making large excursions with a typical shape resembling the cycle observed in the mean field limit, and these excursions occur randomly. Such excursions are typical of the presence of a homoclinic deterministic trajectory: when perturbed, the system catches the homoclinic orbit responsible for such large excursions. The codimension one bifurcation diagram of the 2-neurons system indeed illustrates the presence of a homoclinic orbit as a function of $I_1$ (see diagram~\ref{fig:Codim2Lambda}, and Figure~\ref{fig:BehaviorLambda} (A))\footnote{Indeed, the mean field equations with $\lambda=0$ are precisely the equations of a two-neurons network since in that case $f(\mu, \lambda^2/2)=S(\mu)$.}. Noise can be heuristically seen as perturbing the deterministic value of $I_1$. For sufficiently small values of the noise parameter, the probability of $I_1$ to visit regions corresponding to the presence of a cycle is small. But as the noise amplitude is increased, this probability becomes non-negligible and individual trajectories will randomly follow the stable cycle. Such excursions produce large input to the other neurons which will either be inhibited or excited synchronously at this time, a phenomenon that may trigger synchronized oscillations if the coupling is strong enough and the proportion of neurons involved in a possible excursion large enough. If the noise parameter is too large, the limit cycle structure will be destroyed.

Another way to understand this phenomenon consists in considering the phase plane dynamics of the two-neurons network with no noise (see Figure~\ref{fig:PhasePlane}). The system presents three fixed points, one attractive, one repulsive, and a saddle. The unstable manifold of the saddle fixed point connects with the stable fixed point in an heteroclinic orbit. The stable manifold of the saddle fixed point is a separatrix between trajectories that make small excursions around the stable fixed point, and those related to large excursions close to the heteroclinic orbit. As noise is increased, the probability distribution of each individual neuron, centered around the stable fixed point, will grow larger until it crosses the separatrix with a non-negligible probability, resulting in the system randomly displaying large excursions around the heteroclinic cycle. The fact that a homoclinic path to oscillations is found in the mean field limit can be accounted for by these observations, considering the fact that crossing the separatrix, when noise is of small amplitude, can take an arbitrary long time. The rhythmicity of the oscillations we found and the synchronization are related to the coupling in a complex interplay with the probability of large excursions. These heuristic arguments require a thorough mathematical analysis that is outside the scope of the present paper.

\begin{figure}
	\centering
		\includegraphics[width=.8\textwidth]{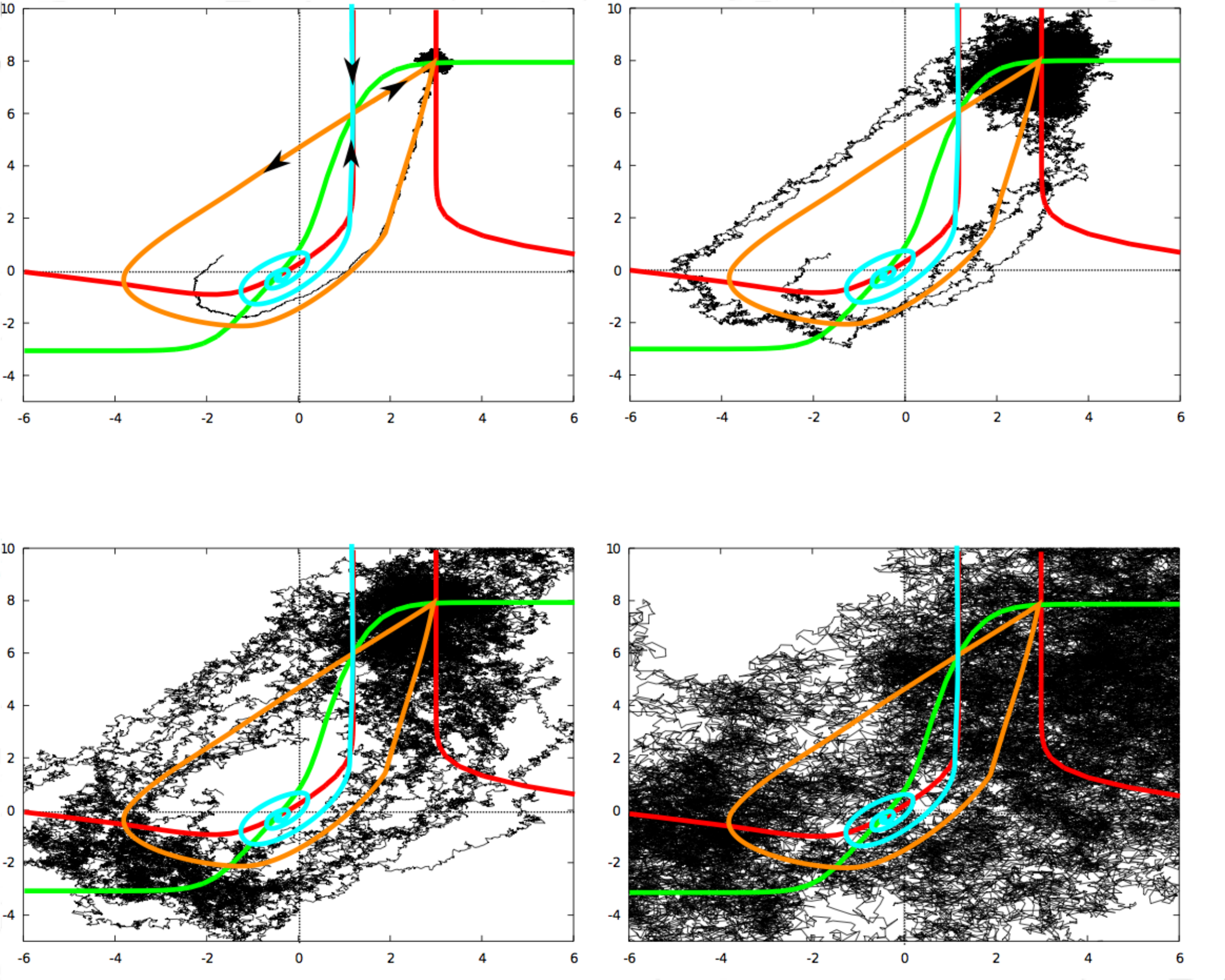}
	\caption{Trajectories in the phase plane for different values of $\lambda$ superimposed on the phase diagram. Red curve: $\mu_1$-nullcline, Green curve: $\mu_2$-nullcline, Orange cycle: unstable manifold of the saddle fixed point (heteroclinic orbit) and Cyan curve: stable manifold of the saddle fixed point (note that it is almost superposed with part of the $\mu_1$-nullcline), constituting the separatrix between those orbits that directly return to the stable fixed point and those following the heteroclinic cycle. Black: noisy trajectories. Upper left: $\lambda=0.2$: no excursion, corresponds to the fixed point regime. Upper right: $\lambda=1$: rare excursions do occur, corresponding to the bistable regime. Bottom left: $\lambda=1.6$: excursions are frequent but occur irregularly  (corresponding to the oscillatory regime). Bottom right: $\lambda=5$: the heteroclinic cycle structure is lost, corresponding to the fixed point regime. }
	\label{fig:PhasePlane}
\end{figure}

\section{Discussion}\label{sec:discussion}

In this article, we have been interested in the large-scale behavior of networks of firing rate neuron models. Using a probabilistic approach, we addressed the question of the behavior of neurons in the network as its size tends to infinity. In that limit, we showed that all neurons behaved independently and satisfied a mean field equation whose solutions are Gaussian processes such that their mean and variance satisfy a closed set of nonlinear ordinary differential equations. Uniform convergence properties were obtained.

We started by studying the solutions of the mean field equations, in particular their dependence with respect to the noise parameter using tools from dynamical systems theory. We showed that the noise had non-trivial effects on the dynamics of the network, such as stabilizing fixed points, inducing or canceling oscillations. A  codimension two bifurcation diagram was obtained when simultaneously varying an input parameter and the noise intensity, as well as simultaneously varying a total input connectivity parameter and the noise intensity. The analysis of these diagrams yielded several qualitatively distinct codimension one bifurcation diagrams for different ranges of noise intensity. Noise therefore clearly induces transitions in the global behavior of the network, structuring its Gaussian activity by inducing smooth oscillations of its mean.

These classes of behaviors were then compared to simulations of the original finite-size networks. We obtained a very good agreement between the simulations of the finite-size system and the solution of the  mean field equations, for networks as small as a few hundreds to few thousands of neurons. Transitions between different qualitative behaviors of the network matched precisely the related bifurcations of the mean field equations, and no qualitative systematic finite-size effects were encountered. Moreover, it appears that the convergence of the solution to a Gaussian process as well as the propagation of chaos property happen for quite small values of $N$, as illustrated in Figure~\ref{fig:Gaussian}. This figure represents the distribution of the voltage potential at a fixed time $T=40$ for $N=500$, simulated for $20$ sample trajectories. The Kolmogorov-Smirnov test validates the Gaussian nature of the solution with a p-value equal to $7\cdot 10^{-4}$. In order to test for the independence, we used the Pearson, Kendall and Spearman tests of dependence. We obtain the correlation values $0.0439$ (p-value $0.33$) for the first population, $0.0212$ (p-value $0.4785$) for the second, and $0.0338$ (p-value $0.45$) for the cross-correlation between populations, all of them clearly rejecting the dependence null hypothesis. This independence has deep implications in the efficiency of neural coding, a concept that we will further develop in a forthcoming paper. 
\begin{figure}
	\centering
		\includegraphics[width=.6\textwidth]{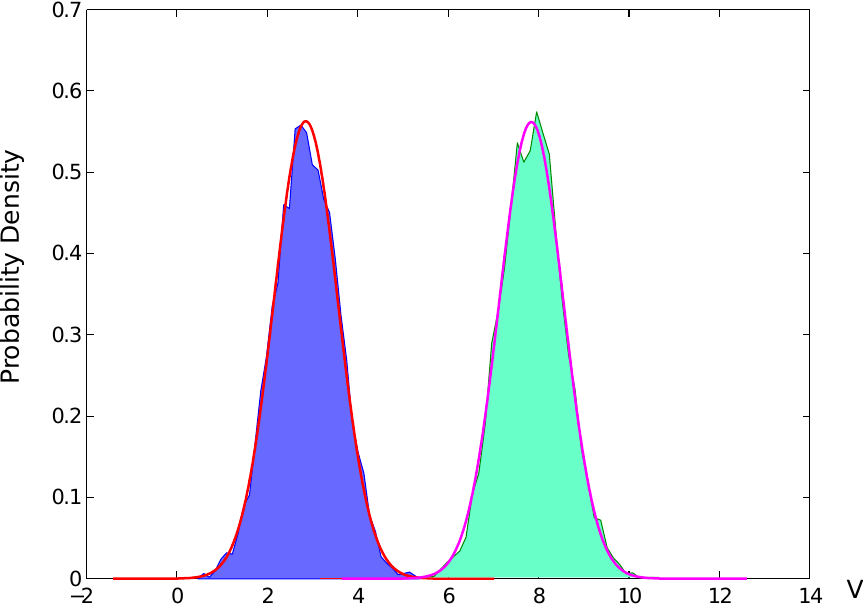}
	\caption{Empirical distribution of the values of $(V^i(T))_{i=1\ldots N}$ for $N=1000$ ($500$ neurons per population) in each population (blue and green filled distribution) versus theoretical mean field distribution. The Kolmogorov-\review{Smirnov} validates the fit of the distributions (see text).}
	\label{fig:Gaussian}
\end{figure}

These findings have several implications in neuroscience and other scientific domains as we now comment.

\subsection{Noise-induced phenomena}
We have seen that the presence of noise in the system induces different qualitative behaviors. For instance, regular oscillations of the mean firing rate, linked with a synchronization of all neurons in the network, appear in the system at some precise values of the noise parameter, in particular for systems that feature a stable fixed point in a noiseless context. This means that noise has a strong structuring effect on the global behavior of a cortical assembly, which is rather a counterintuitive phenomenon, since noise is usually chiefly seen as altering structured responses. This phenomenon adds to a few recent observations made in other settings by various authors. For instance, coherent oscillations in one-population spiking neural networks where exhibited by Pham and collaborators~\cite{pham-pakdaman:98}. The results obtained by Nesse and collaborators~\cite{nesse-bressloff:08} also confirm the presence of noise-induced phenomena in all-to-all coupled networks of leaky integrate-and-fire neurons with filtered noise and slow activity-dependent currents, i.e. noise and spikes are filtered by a second order kernel. The mean field equations they derive are based on a separation of time scale and ergodicity properties. They show in that case the presence of noise-induced burst oscillations similar to the case of~\cite{pham-pakdaman:98}, and a resonance phenomenon. Both approaches differ from our analysis in that they are dealing with spiking neurons. Pham and collaborators are able to reduce the dynamics a discrete-time network of pulse-coupled spike response model neurons, using a Markovian approach, to the study of a discrete time dynamical system. Nesse and collaborators derive their mean field limit using relevant approximations of the Fokker-Planck equations. The interest of our firing-rate model is that it rigorously allows the use of the mathematical theory of propagation of chaos to reduce the system to a set of coupled differential equations that can be studied using the dynamical systems theory.

The phenomena observed in our analysis of large-scale neuronal networks differs from the phenomena of stochastic resonance or coherence resonance well documented in the neuro-computational literature (see e.g.~\cite{lindner:04} for a review of the effect of noise in excitable systems). These phenomena correspond to the fact that there exists a particular level of noise maximizing the regularity of an oscillatory output related to periodic forcing (stochastic resonance) or to intrinsic oscillations (coherence resonance). Such situations are evidenced through the computation of the maximal value of the Fourier transform of the output. In our case, as we can see in the Fourier transform plots, the maximal value of the Fourier transform does not present a clear peak as a function of the noise level (see Figure~\ref{fig:Lambda3d}), hence the system does not exhibit resonance. Besides this observation, the regularity of the oscillation can be expected to be relatively high for large networks in our framework, since the mean activity is asymptotically perfectly periodic. \review{This type of phenomena is fundamentally related to the randomness in the inputs, and will not be observed in the Markovian mean field equations developed by~\cite{buice-cowan:07,bressloff:09}. Indeed, apart from the difference inherent to the fact that they consider Markov chains governing the firing of individual neurons as their microscopic model, the randomness and the correlations in the activity vanishes in the limit $N\to\infty$ yielding the deterministic Wilson and Cowan equation.}

Another important result of ours is the ability to define classes of parameter ranges attached to a few generic bifurcation diagrams as functions of the input to a population. This property suggests further some  reverse-engineering studies allowing to infer from measurements of the system responses to different stimuli the level of noise it is submitted to.

\review{The influence of noise in spiking one-population neural networks was studied in another context by Pham and collaborators in~\cite{pham-pakdaman:98} and Brunel and collaborators~\cite{brunel:00,fourcaud-brunel:02}. In~\cite{pham-pakdaman:98}, the authors study randomly or fully connected one-population networks of spiking neurons. They analyze the probability distribution of spike sequences and reduce this analysis to the study of the properties of a certain map under an independence assumption and in the limit where the number of neurons is infinite (which makes the independence assumption particularly relevant). They show that noise can trigger oscillations for certain values of the total connectivity parameter in a one-population case. Similar phenomena are shown in the study of sparse randomly connected integrate-and-fire neurons as shown in ~\cite{brunel:00} where the system can present synchronous regular regimes. In the mean field model studied in the present article, no oscillatory activity is possible in such one-population systems, since its dynamics can be reduced to a one-dimensional autonomous dynamical system. Smooth nonlinearities in the intrinsic dynamics or discontinuities such as the presence of a spiking threshold in~\cite{pham-pakdaman:98,brunel:00} makes the dynamics of the mean field equations more complex, in particular prevents reduction to a one-dimensional autonomous system governing the mean of the solution. Such intricacies may also be the source of oscillations in one-population systems. }

\review{We eventually emphasize the fact that the noise-induced transitions presented here are related to the nature of the mean field equations, which is not a standard stochastic differential equation. Such phenomena do not generally occur in usual stochastic differential equations, as for instance shown in~\cite{horsthemke-lefever:84}. }

\subsection{Understanding the functional role of noise}
The question of the functional role of  noise in the brain is widely debated today since it clearly affects neuronal information processing. A key point is that the presence of noise is not necessarily a problem for neurons: as an example, stochastic resonance helps neurons detecting and transmitting weak subthreshold signals. Furthermore neuronal networks that have evolved in the presence of noise are bound to be more robust and able to explore more states, which is an advantage for learning in a dynamic environment. 

The fact that noise can trigger synchronized oscillations at the network level enriches the possible mechanisms leading to rhythmic oscillations in the brain, directly relating it to the functional role of oscillations. Rhythmic patterns are ubiquitous in the brain and take on different functional roles. Among those, we may cite visual feature integration~\cite{singer-gray:95}, selective attention, working memory. Abnormal neural synchronization is present in various brain disorders~\cite{uhlhaas-singer:06}. Oscillations themselves can signal a pathological behavior. For instance, epileptic seizures are characterized by the appearance of sudden, collective, slow oscillations of large amplitude, corresponding at the cell level to a synchronization of neurons, and visible at a macroscopic scale through EEG/MEG recordings. This phenomenon is very close to the observation in our model that, as noise is slowly increased, the solutions of the mean field equations undergo a saddle-homoclinic bifurcation abruptly yielding large amplitude and small frequency oscillations. Such a collective phenomenon resembles epileptic seizures. Moreover, in our model and at the microscopic level, these regimes are characterized by a sudden precise synchronization of all neurons in the network, consistent with what is observed at the cell level in epileptic seizures. Our mean field model, based on a simple description of neural activity, was able to account for such complex biologically relevant phenomena, which suggests to use this new model as a cortical mass model and compare it to more established cortical column models such as Jansen and Rit's or Wendling and Chauvel's~\cite{touboul-faugeras-wendling:10,wendling-chauvel:08,jansen-rit:95}.

\subsection{Perspectives}
Several extensions of the present study with applications in neuroscience and in applied mathematics are envisioned. 

The first is to expand our work to include more biologically relevant models and to study the behavior of the solutions of the mean field equations in that mathematically much more complex setting that would include in particular nonlinear intrinsic dynamics and different ionic populations. Another important direction in the development of this work would consist in fitting the microscopic model to biological measurements. This would yield a new neural mass model for large scale areas and develop studies on the appearance of stochastic seizures and rhythmic activity in relationship with different parameters of the model, integrating the presence of noise in a mathematically and biologically relevant manner. 

From this point of view, the simplicity of the model, specifically the linearity of the intrinsic dynamics of each individual neuron, made possible an analytic and quantitative analysis of our mean field equations, a particular case of McKean-Vlasov equations. The drawback of this simplicity is that it does not represent precisely the activity of individual neurons. \review{Rate models are often considered valid at the macroscopic level as describing populations activity, and as such might not be good models of single cells. However, defining the instantaneous firing rate as a trial average~\cite{gerstner-kistler:02,dayan-abbott:01} can be more relevant from this viewpoint. Alternatively, our model can be seen as a model of a hypercolumn, each diffusion process characterizing the activity of a whole cortical column modeled by Wilson and Cowan equations.} For realistic individual neuron models the solutions to the mean field equations will not, in general, be Gaussian. General approaches to study them  consist either in studying their properties as random processes, or in describing their probability distribution. In the first case, one is led to investigate an implicit equation in the space of stochastic processes, and in the second case, one is led to study a complex non-local partial differential equation (Fokker-Planck), \review{as done in a recent paper by Caceres and collaborators~\cite{caceres-carrillo:11}}. In both cases one faces a difficult challenge, and the dependency of the solutions with respect to parameters is extremely hard to describe. 

Another interesting improvement of the model would consist in considering that the synaptic weights are random. These weights can be randomly drawn in a distribution and frozen during the evolution of the network. The propagation of chaos would again applies in that case, and in the rather simple case discussed in the present manuscript, the solution is a Gaussian process as proved in~\cite{faugeras-touboul-etal:09}. However, the mean and covariance cannot be described by a set of ordinary differential equations, as in here since the covariance  depends on the whole previous history of the correlations. Alternatively, the weights can be considered as stochastic processes themselves.  \review{In this case, noise will appear multiplicatively and new noise-induced transitions might appear.} The question of the synaptic noise model chosen and the dynamics of such systems will be addressed in a forthcoming paper.

The present study can be seen as a proof of concept, and it seems reasonable to extrapolate that such noise-induced transitions do occur as well for the solutions to the mean field equations of these more complex and more biologically plausible systems.

\appendix

\section{Proof of lemma \ref{lemma:ErfSigmoids}}\label{app:proof}

In this appendix we prove lemma \ref{lemma:ErfSigmoids} stating that in the case where the sigmoidal transforms are of the form $S_{\alpha}(x) = \erf(g_{\alpha}x+\gamma_{\alpha})$, the functions $f_{\alpha}(\mu_{\alpha},v_{\alpha})$ involved in the mean field equations \eqref{eq:MFE} with a Gaussian initial condition take the simple form~\eqref{eq:M-Erf}.

\noindent
\begin{proof}
We have, using the definition of the $\erf$ function
\begin{align*}
	  \Exp{S_{\alpha}(X_{\alpha})(t)}&=\int_{\R} \erf\left(g_{\alpha}\left(x \sqrt{v_{\alpha}(t)}+\mu_\alpha(t)\right)+\gamma_{\alpha}\right) \frac{e^{-x^{2}/2}}{\sqrt{2\pi}}\,dx\\
	  & = \int_{\R}\int_{-\infty}^{g_{\alpha}\left(x \sqrt{v_{\alpha}(t)}+\mu_\alpha(t)\right)+\gamma_{\alpha}}\frac{e^{-(x^{2}+y^{2})/2}}{2\pi} dx dy
	\end{align*}
	This integral is of the form:
	\[\int_{\R}\int_{-\infty}^{a\,x +b}\frac{e^{-(x^{2}+y^{2})/2}}{2\pi} dx dy\]
	and therefore, the integration domain has an affine shape as plotted in figure \ref{fig:ChangeVar}. 
		\begin{figure}[!h]
		\begin{center}
			\includegraphics[width=.3\textwidth]{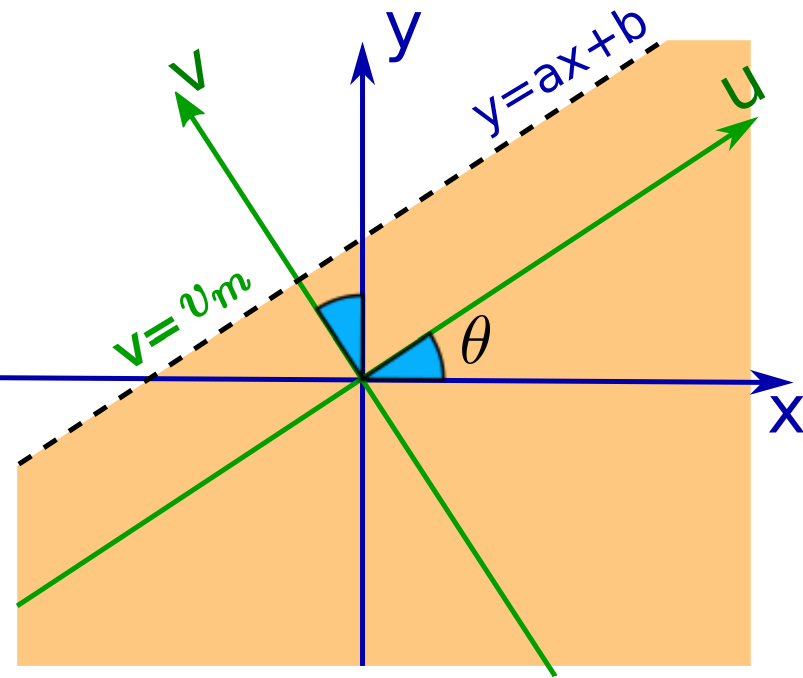}
		\end{center}
		\caption{Change of variable for the erf function.}
		\label{fig:ChangeVar}
		\end{figure} 
	In order to compute this integral, we change variables by a rotation of the axes $(x,y)$ and align the affine boundary of our integration domain with our new variables $(u,v)$ (see figure \ref{fig:ChangeVar}). Simple geometric analysis shows that the rotation angle $\theta$ for this change of variable is such that $\tan(\theta)=a$. The new integration domain is in the new coordinates given by $v\leq v_m=b \cos(\theta) = \frac{b}{\sqrt{1+a^{2}}}$:
	\begin{align*}
	  \int_{\R}\int_{-\infty}^{a\,x +b}\frac{e^{-(x^{2}+y^{2})/2}}{2\pi} dx dy &=\int_{\R}\int_{-\infty}^{\frac{g\,b}{\sqrt{1+g^{2}a^{2}}}} e^{-(u^{2}+v^{2})/2} \frac{1}{2\pi} du dv\\
	  & = \erf\left(\frac{gb}{\sqrt{1+g^{2}a^{2}}}\right)
	\end{align*}
	which reads with the parameters of the model:
	\begin{equation*}
		f_\alpha(\mu,v) = \erf\left(\frac{g_{\alpha}\, \mu + \gamma_{\alpha}}{\sqrt{1+g_\alpha^{2}v}}\right)
	\end{equation*}
\end{proof}

\section{Bifurcations Diagram as a function of $\lambda$}\label{sec:BifLambda}
In section~\ref{sec:bifs}, we observed that six different bifurcation diagrams appear as $I_1$ is varied, depending on the value of $\lambda$ characterizing the additive noise input. For the particular choice of parameters chosen in that section, the different zones are segmented for values of $\lambda$ given in table~\ref{tab:numericalLambda}. 

\begin{table}[!h]
	\begin{center}
	\begin{tabular}{|c|c|c|c|c|c|}
		\hline
		Type & C & BT & Hom TP & H TP & C\\
		\hline
		$\lambda$& 0.16 & 2.934 & 2.948 & 2.968 & 3.74\\
		\hline
	\end{tabular}
	\end{center}
	\caption{Numerical values of the separation into six $\lambda$ zones for Figure~\ref{fig:Codim2Lambda}. C stands for Cusp, BT: Bogdanov-Takens, Hom TP: turning point of the Homoclinic bifurcations curve, H TP: Hopf bifurcation curve turning point. }
	\label{tab:numericalLambda}
\end{table}

In each of these zones, typical codimension 1 bifurcation diagrams as the input $I_1$ is varied are depicted in figure~\ref{fig:BehaviorLambda}. We now describe the behavior of the system in each of these zones.
\begin{figure}
	\centering
		\includegraphics[width=.7\textwidth]{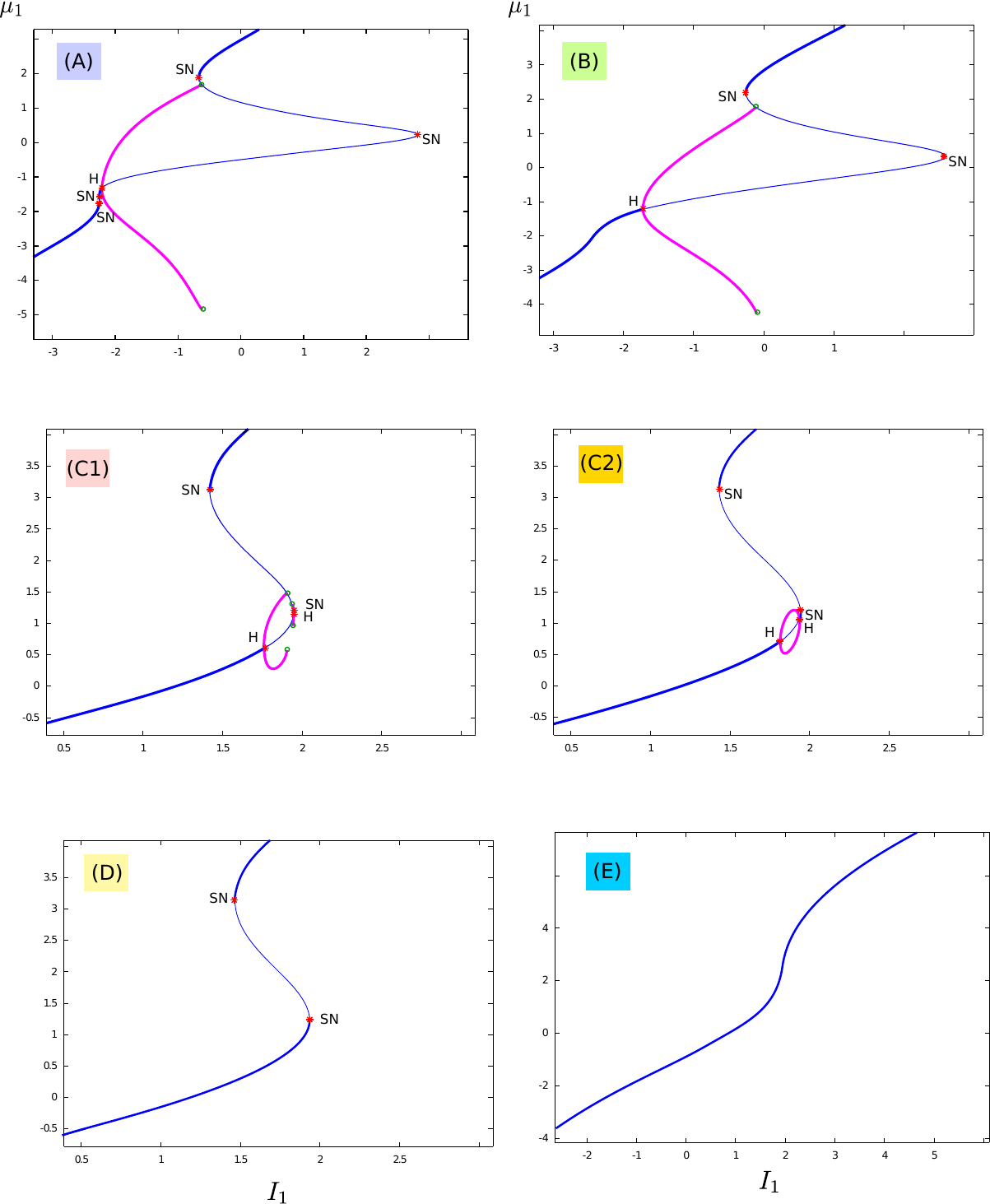}
	\caption{Typical behavior of the system in each zone (A) through (F). (A): $\lambda=0$, (B): $\lambda=1$, (C): $\lambda=2.945$, (D):$\lambda=2.955$, (E): $\lambda=3$, (F): $\lambda=4$. Red stars: bifurcations, SN: Saddle-Node, H: Hopf, green circle: Saddle-Homoclinic bifurcation. Thick blue line: stable fixed point, thin blue line: unstable fixed point, thick pink line: stable cycle. See text for precise description.}
	\label{fig:BehaviorLambda}
\end{figure}

% \textcolor{blue}{{\bf Copie Colle du texte principal}Zone (C) and (D) fe is a zone where the system is very sensitive to the variations of $\lambda$. In this zone, the Hopf and the homoclinic bifurcation curves (arising from the Bogdanov Takens bifurcation) have a turning point. Therefore, in this very small parameter zone corresponding to $\lambda \in [2.935,2.965]$, the system can either have two Hopf bifurcations and two saddle homoclinic bifurcations (C), i.e. two disconnected branches of cycles that appear through supercritical Hopf bifurcations and disappear through homoclinic bifurcations, or two supercritical Hopf bifurcations whose family of limit cycles are identical (D). The different behaviors are displayed and described in appendix \ref{sec:BifLambda}.}

\begin{description}
	\item[(A)] For very small values of $\lambda$, the system features four saddle-node bifurcations and one supercritical Hopf bifurcation, associated to the presence of stable limit cycles that disappear through saddle-homoclinic bifurcation arising from the Bogdanov-Takens bifurcation (after the turning point.) In an extremely limited range of parameter, the occurrence of two saddle-node bifurcation relates to a bistable regime  in that small parameter region.
	\item[(B)] In zone (B), the system differs from zone (A) in that the two inferior saddle-node bifurcation disappeared through Cusp bifurcation. Globally the same behavior are observed, except for the bistable behavior commented above (which was not a prominent phenomenon due to the reduced parameter region concerned). 
	\item[(C)] On the upper branch of saddle nodes, the systems undergoes a Bogdanov-Takens bifurcations, yielding the presence in zone (C) of a supercritical Hopf bifurcation and of a saddle-homoclinic bifurcation curve. This BT bifurcation accounts for the family of Hopf bifurcations observed in zones (A-B) and for the saddle-homoclinic bifurcations, because of the turning points observed in the full bifurcation diagrams. In region (C), two families of limit cycles coexist, both arising from supercritical Hopf bifurcation and disappearing through saddle-homoclinic bifurcation. 
	\item[(D)] Because of the topology of the bifurcation diagram, the turning point of the saddle-homoclinic bifurcations curve occurs before the turning point of the Hopf bifurcations curve. This difference yields zone (D) between the two turning points. In that zone, we still have two supercritical Hopf bifurcations, but no more homoclinic bifurcation. The families of limit cycles corresponding to each of the Hopf bifurcations are identical. 
	\item[(E)] After the turning point of the Hopf bifurcations manifold, we are left with two saddle-node bifurcations, hence a pure bistable behavior with no cycle.
	\item[(F)] Both saddle-node bifurcation disappear by merging into a cusp bifurcation. After this cusp, the system has a trivial behavior, i.e. it features a single attractive equilibrium whatever $I_1$. 
\end{description}

% \newpage

\bigskip

\noindent
{\bf Acknowledgements}\\
This work was partially supported by the ERC grant \#227747 NerVi.

\bibliographystyle{siam}%plain
% \newpage

\end{document}